%% file: aposteriori_error_forward.tex
\documentclass[final,leqno]{article}
\usepackage{amsfonts,bm}
\usepackage{graphicx}
\usepackage{amsmath}
\usepackage{setspace}
\usepackage{float}
\usepackage{subfigure,color,epsfig,multirow,bbm,graphicx,graphics, sidecap}
\usepackage{tabularx,ulem}
\usepackage{algorithm} 
\usepackage{algpseudocode}
\usepackage{xcolor}
\usepackage[numbers,sort&compress]{natbib}
\usepackage{algorithm}
\usepackage{algpseudocode,ulem}
\usepackage{authblk}
\input{notation_Razvan.tex}

\renewcommand{\tx}{\mathbf{\tilde{x}}}

\newcommand{\x}{\mathbf{x}}

\newcommand{\F}{\mathbf{F}}

\newtheorem{theorem}{Theorem}
\newtheorem{proof}{Proof}
\providecommand{\keywords}[1]{\textbf{\textit{Keywords---}} #1}

\newcommand{\M}{\mathcal{M}}

\newtheorem{remark}{Remark}

\title{A Goal-Oriented Adaptive Discrete Empirical Interpolation Method}

\author{
 R. \c Stef\u anescu\\
  \texttt{GVM Model Department, Spire Global, Boulder, CO 80302, USA, razstefanescu@gmail.com}
  \and
  A. Sandu\\
  \texttt{Computational Science Laboratory, Department of Computer Science, Virginia Polytechnic Institute and State University, Blacksburg, Virginia 24060, USA, sandu@cs.vt.edu}
}

\begin{document}
\maketitle

\begin{abstract}
In this study we propose a-posteriori error estimation results to approximate the precision loss in quantities of interests computed using reduced order models. To generate the surrogate models we employ Proper Orthogonal Decomposition and Discrete Empirical Interpolation Method. First order expansions of the components of the quantity of interest obtained as the product between the components gradient and model residuals are summed up to generate the error estimation result. Efficient versions are derived for explicit and implicit Euler schemes and require only one reduced forward and adjoint models and high-fidelity model residuals estimation. Then we derive an adaptive DEIM algorithm to enhance the accuracy of these quantities of interests. The adaptive DEIM algorithm uses dual weighted residuals singular vectors in combination with the non-linear term basis. Both the a-posteriori error estimation results and the adaptive DEIM algorithm were assessed using the 1D-Burgers and Shallow Water Equation models and the numerical experiments shows very good agreement with the theoretical results.
\end{abstract}

\keywords{proper orthogonal decomposition; discrete empirical interpolation method; a-posteriori error estimates; shallow water equations}
%

\pagestyle{myheadings}
\thispagestyle{plain}
\markboth{A-posteriori Error Estimates for Adaptivity of Reduced Order Models}{R. \c Stef\u anescu and A. Sandu}
\section{Introduction}\label{sec:Intro}

The major issue in large scale complex modelling is that of reducing the computational cost while preserving numerical accuracy. Among the model reduction
techniques, reduced basis \cite{BMN2004,grepl2005posteriori,patera2007reduced,rozza2008reduced,Dihlmann_2013} and Proper Orthogonal Decomposition (POD) \cite{karhunen1946zss,loeve1955pt,hotelling1939acs,lorenz1956eof}
provide an
efficient means of deriving the reduced basis for high-dimensional non-linear flow systems. Data analysis is conducted to extract basis functions, from experimental data or
detailed simulations of high-dimensional systems, for subsequent use in Galerkin or Petrov-Galerkin projections that yield low dimensional dynamical models.

However due to the model non-linearities the computational complexity of the reduced order models still depends on the number
of variables of the high-fidelity model. The current literature presents several ways to avoid this issue such as the empirical interpolation method
(EIM) \cite{BMN2004} and its discrete variant DEIM \cite{Cha2008,ChaSor2012,ChaSor2010}, best points interpolation method \citep{NPP2008}. Missing point estimation
\citep{Astrid_2008} and Gauss-Newton with approximated tensors
\citep{Carlberg2_2011,Carlberg_2012} methods are relying upon the gappy POD technique \citep{Everson_1995} and were developed for
the same reason.  A comparative study between missing point estimation method, gappy POD and DEIM is available in \cite{dimitriu2017comparative}.

A-posteriori error estimates of DEIM reduced non-linear dynamical system based on logarithmic Lipschitz constants are available in the literature \cite{wirtz2014posteriori}
. Additional state space error estimates \cite{ChaSor2012} are shown to be proportional to the sums of the singular values corresponding to the neglected POD basis vectors both in Galerkin projection
of the reduced system and in the DEIM approximation of the non-linear term.

The parameter reduced order modeling poses even more difficulties since one may require to construct a global basis to enable both accurate and fast reduced order
models along the entire parametric domain. Recently adaptive techniques capable to enhance the accuracy of the reduced order models or quantities of interest
depending on such surrogate models have been proposed. For example, by assigning
larger weights to samples that are more important or have a higher probability of occurring allow the construction of smaller global bases
for both state and non-linear terms \cite{chen2013weighted,chen2014weighted}. In the data assimilation field, the sensitivities of a cost functional
with respect to the time-varying model state is used to define appropriate weights and to implement a dual weighted POD method for order reduction \cite{daescu2008dual}. Effective exploration of the parameter space by adaptive grids based
on a-posteriori error estimates \cite{haasdonk2008adaptive} brought improvement over the fixed and uniformly refined grid approaches. Dual techniques \cite{veroy2005certified} for a-posteriori error estimates enabled the construction of a goal-oriented global parametric basis.
The idea of a spanning ROM \cite{weickum2006multi} suggests  interpolating the associated high-fidelity solutions for a new parameter configuration
and then extracting the singular vectors.

To generate reduced order models being simultaneously accurate and performant local strategies have been designed. The parametric or time domains
are partitioned into several sub-regions and local bases are constructed for both the state variables and non-linear terms. Local reduced operators associated with the EIM \cite{eftang2012parameter} or DEIM \cite{peherstorfer2014localized} approximations and POD \cite{Rapun_2010} or
reduced basis models \cite{dihlmann2011model} are computed off-line and are properly selected during the on-line phase. Dictionary methods \cite{kaulmann2013online,maday2013locally}
instead of building the reduced basis during the off-line phase construct a small parameter adapted basis using a Greedy procedure during the on-line stage. A different approach solves a
residual minimization problem to obtain a low dimensional approximation of an incremental solution on-line \cite{amsallem2012nonlinear} without
the need of computing reduced operators during the off-line stage.

Another form of adaptivity involves interpolation of already existing bases in the matrix space \cite{lieu2004parameter}, principal vector space
\cite{lieu2004parameter} and in the space tangent to the Grassmann manifold \cite{Amsallem_2008} to generate a reduced order
basis for a new parameter configuration. However these techniques still require computation of reduced operators thus they can be classified as off-line
approaches. To mitigate this inefficiency directly interpolating the system matrices of the local reduced models \cite{panzer2010parametric} has been proposed
in the space tangent to the Grassmann \cite{amsallem2011online,Zimmermann2014}  and Riemannian \cite{degroote2010interpolation} manifolds.

Several approaches that incorporate new data online and rebuild the reduced system are available in the reduced order optimization field \cite{EArian_MFahl_EWSachs_2000,
weickum2009multi, cstefuanescu2015pod, zahr2015progressive} and parametric reduced order modeling \cite{paul2015adaptive}.

In this paper we introduce a-posteriori error estimation results to approximate
the errors in quantities of interests resulted from using reduced
POD/DEIM reduced order models. First order expansions of the components of the quantity of interest obtained as the
product between the components gradient and model residuals are summed
up to generate the error estimation result. Efficient versions are derived for explicit and implicit Euler schemes and require only one reduced forward and adjoint models and high-fidelity model residuals estimation.

Later, we propose a novel strategy that incorporates on-line dual weighted residuals associated with some quantity of interest and modifies the location of the DEIM
points corresponding to the DEIM non-linear reduced order terms approximations. In this way we enhance the accuracy of the quantity of interest computed via
reduced order models. We show that our methodology works also for different parameter configurations as seen in the numerical experiments section using 1D-Burgers
and Shallow Water Equations models. The key ingredients are: (1) an a-posteriori error estimation result based on necessary optimality conditions of a constrained
optimization problem; (2) an adaptive greedy algorithm that makes use of both non-linear terms and dual weighted residuals bases, respectively.

In contrast to the work proposed in \cite{Peherstorfer_Willcox_2015} where both non-linear reduced basis and DEIM interpolation points are adapted with additive low-rank updates, we are only adapting the location of the interpolation points by taking
into account the dual weighted residuals. Our approach follows the research developed in \cite{meyer2003efficient} where dual weighted method is used to adaptively resize the number of basis vectors and the length of the time step to satisfy a given error tolerance.

This paper is organized as follows. Section \ref{sec:ROM} reviews the Proper Orthogonal Decomposition and Discrete Empirical Interpolation methods and Section \ref{sec:problem_formulation} formulates the problem of constructing reduced order models with adaptive DEIM interpolation points. In Section \ref{sec:aposteriori} we introduce the a-posteriori error estimation approach for the error in a quantity of interest. Fast computational strategies based on a reduced order adjoint model are derived and enable the use of the a-posteriori error estimations for different parametric configurations. Section \ref{sec:DEIM_adapt} presents a greedy adaptive algorithm that uses the dual weighted residuals to update the location of DEIM interpolation points. Section \ref{sec:numerical} presents numerical results with the 1D-Burgers and Shallow Water Equations models. In Section 8 we conclude the paper with some remarks.

\section{Reduced order modeling}\label{sec:ROM}

Reduced order modeling represents the dynamics of large-scale systems using only a smaller number of variables and reduced order basis functions. We consider here the construction of reduced order models via Proper Orthogonal Decomposition (POD)
and Galerkin projection. To mitigate a known deficiency of POD Galerkin, we make use of Discrete Empirical Interpolation Method to approximate the non-linear terms.

\subsection{Proper Orthogonal Decomposition}\label{sec:POD}

Proper Orthogonal Decomposition (Karhunen-Lo\`{e}ve decomposition) has been proposed in \cite{karhunen1946zss,loeve1955pt,Kosambi1943,Obukhov1954} for an infinite dimensional framework.  Latter on, the finite dimensional case leads to the development of principal component analysis \cite{jolliffe2002principal} in the statistical literature, inspired from the early work of \citet{pearson1901} and \citet{hotelling1939acs}. The approach is named empirical orthogonal function decomposition in oceanography and meteorology \cite{lorenz1956eof}, and factor analysis in psychology and economics \cite{gorsuch1990common}.

POD can be thought of as a Galerkin approximation in the spatial variable built from functions corresponding to the solution of the physical system at specified time instances, as obtained by the method of snapshots \cite{Sir87a,Sir87b,Sir87c}. POD has been used successfully in numerous applications such as unsteady viscous fl‚ows \cite{epureanu2001reduced}, compressible flows \citep{Rowley2004}, computational fluid dynamics \citep{Kunisch_Volkwein_POD2002,Rowley2005,Willcox02balancedmodel,Noack2010}, turbulent flows \cite{San_Iliescu2013,wells2015regularized}, and aerodynamics  \citep{Bui-thanh04aerodynamicdata}.

We consider a general discrete dynamical system
\begin{equation}
\label{eqn:full_forward_model_init}
 {\mathbf x}_{i+1} =  \M_{i,i+1}\left({\mathbf x}_i,\mu\right), \quad i=0,..,\Nt-1, \quad \mu \in \mathcal{\tilde P},
\end{equation}
where $\x_i \in \mathbb{R}^{\Ns}$, $i=0,\ldots,\Nt$, denotes the state at time $t_i$,  with $\Ns$ being the total number of discrete model variables per time step, and $N_t \in \mathbb{N},~N_t>0$ being the number of time steps. Here $\mu$ is a vector of parameters that characterizes the physical properties of the flow,  e.g., the Coriolis force in a Shallow Water Equations model.

For a given parameter configuration we select an ensemble of $N_t$ time instances (snapshots) ${\bf x}_{t_1}^{\mu},...,{\bf x}_{t_{N_t}}^{\mu}  \in \mathbb{R}^{\Ns}$ from the solution of \eqref{eqn:full_forward_model_init}. The POD method chooses an orthonormal basis $U_{\mu}=[{\bf u}_{i}^{\mu}]_{i=1,..,k} \in \mathbb{R}^{{\Ns}\times k}$ such that the mean square error between ${\bf x}(\mu,t_i)$ and the projection ${\bf x}_\textsc{pod}^{\mu}(t_i) = U_{\mu}\,{\bf \tilde x}(\mu,t_i),$ ${\bf \tilde x}(\mu,t_i) \in \mathbb{R}^k,~i=1,..,N_t,$ is minimized on average.  The POD space dimension $k \ll {\Ns}$ is appropriately chosen to capture the dynamics of the flow as described by Algorithm \ref{euclid}.

\begin{algorithm}
 \begin{algorithmic}[1]
 \State Compute the singular value decomposition for the snapshots matrix $[{\bf x}_{t_1}^{\mu}~~ ... ~~{\bf x}_{t_{N_t}}^{\mu}] = \bar U_{\mu} \Sigma_{\mu} {\bar V}^T_{\mu},$
 with the singular vectors matrix $\bar U_{\mu} =[{\bf u}_i^{\mu}]_{i=1,..,\Ns}.$
 \State Using the singular-values $\lambda_1\geq \lambda_2\geq ...\lambda_n\geq 0$ stored in the diagonal matrix $\Sigma_{\mu}$, define   $I(m)=\left(\sum_{i=1}^m \lambda_i\right) / \left(\sum_{i=1}^{\Ns} \lambda_i\right)$.
\State Choose the dimension of the POD basis as $ k=\arg\min_m \{I(m):I(m)\geq \gamma\}$, where $0 \leq \gamma \leq 1$ is the percentage of total information captured by the reduced space $U_{\mu} = \textrm{span}\{{\bf u}_1^{\mu},{\bf u}_2^{\mu},...,{\bf u}_k^{\mu}\}.$ Usually $\gamma=0.99$.
 \end{algorithmic}
 \caption{POD basis construction}
 \label{euclid}
\end{algorithm}
Another way to compute the POD basis is to make use of the eigenvalue decomposition applied to the correlation matrix \cite[Alg. 1]{stefanescu2014comparison}.  However the SVD-based POD basis construction is more computationally efficient.

The Galerkin approach projects the full order state $\mathbf{x}$ and the full order model equations \eqref{eqn:full_forward_model_init} onto the space spanned by the POD basis elements to obtain the following reduced model:
\begin{equation}
\label{eqn:reduced_forward_model_init}
 \tx_{i+1} = U_{\mu}^T \cdot \M_{i,i+1}\left(U_{\mu}\,\tx_{i},\mu\right), \quad i=0,..,\Nt-1, \quad \mu \in \mathcal{\tilde P}.
 \end{equation}
%

The efficiency of the POD-Galerkin techniques is limited to linear or bilinear terms, since the projected
non-linear terms at every discrete time step still depend on their evaluation in the full model space, e.g.,
\begin{equation}
\tilde N({\bf \tilde x}) \approx \underbrace{U_{\mu}^T}_{~k\times \Ns} \cdot \underbrace{\F(U_{\mu}\tx_i,\mu)}_{\Ns \times 1},
\label{eqn::POD_nonlinearity}
\end{equation}
where $\F : \mathbb{R}^{\Ns} \to \mathbb{R}^{\Ns}$ is the non-linear component of the model \eqref{eqn:full_forward_model_init}. In case of polynomial non-linearities, the tensorial POD technique \cite{stefanescu2014comparison} can be employed to remove the dependence on the dimension of the full order system by manipulating the order of computing.

To mitigate this inefficiency we make use of the  Discrete Empirical Interpolation Method  \cite{ChaSor2010,Stefanescu2013} that can handle efficiently any type of non-linearity and provides a  natural framework for adaptive reduced order modeling.

\subsection{Discrete Empirical Interpolation Method}\label{sec:DEIM}
The empirical interpolation method \cite{BMN2004,MBarrault_YMaday_NDNguyen_ATPatera_2004a}  and its discrete version DEIM \cite{ChaSor2010} were designed to estimate non-linear terms allowing  for an effectively affine offline-online computational decomposition in the context of reduced order models. Both interpolation methods provide an efficient way to approximate non-linear functions in the continuous and discrete frameworks. They were successfully used in the POD framework with finite difference,
finite element, and  finite volume discretization methods. A description of EIM in connection with the  reduced basis framework and a-posteriori error bounds can be found in \citep{Mad_Nguy_Pat_Pau,Grep_Mad_Nguy_Pat}.

The DEIM implementation is based on a POD approach combined with a greedy algorithm. For $m\ll \Ns$, the POD/DEIM approximation of \eqref{eqn::POD_nonlinearity}
is
\begin{equation}
\tilde N({\bf \tilde x}) \approx \underbrace{U_{\mu}^TV_{\mu}(P^TV_{\mu})^{-1}}_{{\rm precomputed}~k\times m}
\cdot  \underbrace{P^T\F\left(U_{\mu}{\bf \tilde x},\mu\right)}_{m  \times 1}
\label{eqn::POD_DEIM_nonlinearity}
\end{equation}
where $V_{\mu}\in \mathbb{R}^{\Ns \times m}$ collects the first $m$ POD basis modes of non-linear function ${\bf F}$, while $P \in \mathbb{R}^{\Ns \times m}$ is the DEIM interpolation selection matrix. The core of the DEIM procedure is an iterative greedy procedure given in Algorithm \ref{alg::DEIM} that inductively constructs $P$ from the linearly independent set $V_{\mu}$ \citep[section 3]{ChaSor2010}. Specifically, Algorithm \ref{alg::DEIM} determines the indices $\rho_1,\rho_2,...,\rho_m$, and constructs $P=[e_{\rho_1},..,e_{\rho_m}]\in \mathbb{R}^{n\times m}$, where $e_{\rho_i}\in\mathbb{R}^n$ is the $\rho_i$-th column of the identity matrix.

\begin{algorithm}
{\bf INPUT}: $\{{\bf v}_\ell\}_{\ell=1}^m\subset\mathbb{R}^{\Ns}$ (linearly independent): \newline
{\bf OUTPUT}: $\bm{\rho}=[\rho_1,..,\rho_m]\in\mathbb{N}^m$
 \begin{algorithmic}[1]
 \State $\{\psi,\rho_1\}=\max|{\bf v}_1|; {\psi}\in \mathbb{R}$ is the largest absolute value entry of $v_1$, and $\rho_1$ is its position, with the smallest index taken in case of a tie.
 \State $V_{\mu}:=[{\bf v}_1]\in \mathbb{R}^{\Ns},~P:=[e_{\rho_1}]\in \mathbb{R}^{\Ns},~\bm{\rho}:=[\rho_1]\in \mathbb{N}.$
 \For{$\ell=2,..,m$}
      \State Solve $(P^TV_{\mu})\, {\bf c}=P^T{\bf v_l} \textrm{  for } {\bf c}\in \mathbb{R}^{\ell-1};\quad V_{\mu},P\in\mathbb{R}^{{\Ns}\times(\ell-1)}.$
      \State ${\bf r}:={\bf v}_\ell-V_{\mu}{\bf c},~{\bf r}\in \mathbb{R}^{\Ns}.$
      \State $\{\psi,\rho_1\}=\max\{|{\bf r}|\}.$
      \State $V_{\mu} := [V_{\mu}~~{\bf v}_\ell],~~P := [P~~e_{\rho_l}],~~\bm{\rho} := \left[
           \bm{\rho}^T~~\rho_\ell\right]^T.$
   \EndFor
 \end{algorithmic}
 \caption{Standard DEIM algorithm for obtaining interpolation indices}
 \label{alg::DEIM}
\end{algorithm}

The algorithm first searches for the largest absolute value entry of the first POD basis $|{\bf v}_1|$, and the corresponding index represents the first DEIM interpolation index $\rho_1\in\{1,2,..,n\}$. The remaining interpolation indices $\rho_\ell,~\ell=2,3..,m$, are selected so that each of them corresponds to the entry with the largest absolute value in $|{\bf r}|$. The vector ${\bf r}$ can be viewed as the residual or the error between the next basis vector ${\bf v}_\ell,~l=2,3..,m$ and its approximation $V_{\mu}{\bf c}$ from interpolating the previous basis vectors $\{{\bf v}_1,{\bf v}_2,..,{\bf v}_{\ell-1}\}$ at the previously determined indices ${\rho_1},{\rho_2},..,{\rho_{\ell-1}}$. The linear independence of the vectors $\{{\bf v}_\ell\}_{\ell=1}^m$ guarantees that ${\bf r}$ is a nonzero vector at each iteration, and that the output indices $\{\rho_\ell\}_{\ell=1}^m$ are not repeating.

The spectrum analysis of the non-linear snapshots correlation matrix offers guidance to choosing the number of DEIM interpolation points. We are particularly interested in the eigenvalues rate of descent. Lemma $(3.2)$ in \citep[section 3.2]{ChaSor2010}
provides an error bound for the DEIM approximation \eqref{eqn::POD_DEIM_nonlinearity} that can be estimated by the largest POD eigenvalue of the snapshots correlation matrix not taken into account by POD basis $V_{\mu}$.

Recent developments of DEIM include rigorous state space error bounds \citep{ChaSor2012}, a-posteriori error estimation \citep{Wirtz2012}, and applications to 1D FitzHugh-Nagumo model \citep{ChaSor2010}, 1D simulating neurons model \citep{ChaSor2_2010}, 1D non-linear thermal model \citep{HBW2011}, 1D Burgers equation \citep{Aanonsen2009,Cha2008}, 2D non-linear miscible viscous fingering in porous medium \citep{ChaSor2011}, oil reservoirs models \citep{EkaSuwartadi2012}, and 2D Swallow Water Equations model (SWE) \citep{Stefanescu2013}. We emphasize that only few POD/DEIM studies with finite elements or finite volume methods have been performed, e.g., for electrical networks \citep{HK2012} and for a 2D ignition and detonation problem \citep{Nguyen_Van_Bo}. Flow  past a cylinder simulations using a hybrid reduced approach combining the quadratic expansion method and DEIM are performed in \citep{Xiao2014}.

\section{Problem formulation}\label{sec:problem_formulation}

We are interested  in a particular aspect of  the solution of \eqref{eqn:full_forward_model_init} defined
by the smooth scalar function $\mathcal{Q}(\x,\mu)$ with
\begin{equation}
\label{eqn:hfqoi}
 \mathcal{Q}: \mathbb{R}^{\Ns \times \mathcal{\tilde P}} \to \mathbb{R}, \mathcal{Q}\left(\x_0,\mu\right) =\sum_{i=0}^{\Nt}~r_i\left(\x_i,\mu\right),
\end{equation}
with $r_i:\mathbb{R}^{\Ns \times \mathcal{\tilde P}} \to \mathbb{R},~i=0,\ldots,\Nt$.
We call \eqref{eqn:hfqoi} the quantity of interest (QoI).

A POD/DEIM reduced order model described in \eqref{eqn:reduced_forward_model_init} and \eqref{eqn::POD_DEIM_nonlinearity} is capable to decrease the computational complexity of evaluating $\mathcal{Q}$ in \eqref{eqn:hfqoi} by orders of magnitudes. However the reduced order approximation leads to an error in the computed QoI denoted by
\begin{equation}
\label{eqn::error_qoi}
 \varepsilon(\mu) = \mathcal{Q}\left(\x,\mu \right) - \mathcal{Q}\left(\widehat{\bx},\mu\right),
\end{equation}
where $\widehat{\bx} = \{\widehat{\bx}_0,~\widehat{\bx}_1,..,\widehat{\bx}_{\Nt}\}$ is the reduced order solution projected onto the full space, i.e. $\widehat{\bx}_i = U_{\mu}\,\tx_i$, $i=0,..,\Nt$.

We seek to construct adaptive reduced order models that allow to accurately estimate the QoI \eqref{eqn:hfqoi}. Specifically, we propose a procedure to adaptively modify the locations of DEIM interpolation points associated with the model non-linear term \eqref{eqn::POD_DEIM_nonlinearity} such as to decrease the error in the QoI \eqref{eqn::error_qoi}. Our approach uses of an efficient a-posteriori error estimate based on the solution of a reduced order adjoint model, and the residual of the high-fidelity model computed with the projected reduced order solution. The dual weighted residuals (the Hadamard products of the projected reduced order adjoint solution and the model residuals) are used to design the new locations for the DEIM points to increase the accuracy of $\mathcal{Q}\left(\widehat{\bx},\mu\right)$.

\section{A-posteriori error estimates}\label{sec:aposteriori}

\subsection{Gradient of the quantity of interest}\label{sec:Lagrangian}

Our a-posteriori error estimate framework requires the first order necessary optimality conditions of the following constrained optimization  problem

\begin{subequations}
\label{eqn::Qoi_optimization}
\begin{equation}
\label{eqn::Qoi_optim}
 \min_{\x_0}~~\mathcal{Q}\left(\x_0,\mu\right) = \sum_{i=0}^{\Nt}~r_i\left(\x_i,\mu\right),
\end{equation}
subject to the constraints posed by the non-linear forward model dynamics \eqref{eqn:full_forward_model_init}
\begin{equation}
\label{eqn:full_forward_model}
 {\mathbf x}_{i+1} = \M_{i,i+1}\left({\mathbf x}_i,\mu\right), \quad i=0,..,\Nt-1.
\end{equation}
\end{subequations}

\begin{theorem}[Minimization of the QoI]
\label{th::first_order_nc}
Assume that the model operators $\M_{i,i+1} : \mathbb{R}^{\Ns \times \mathcal{\tilde P}} \to \mathbb{R}^{\Ns},~i=0,..,\Nt-1,$ are of class $C^1$, and the scalar functions $r_i:\mathbb{R}^{\Ns \times \mathcal{\tilde P}} \to \mathbb{R},~i=0,..,\Nt,$ belong to the class $C^1$ in the first variable. Then the first order necessary optimality conditions of the problem \eqref{eqn::Qoi_optimization} are given by:
\begin{subequations}
\label{eqn:KKT_Full_forward_model}
\begin{equation}
\label{eqn:Full_forward_model}
\textnormal{{Forward model}:  } {\mathbf x}_{i+1} = \M_{i,i+1}\left({\mathbf x}_i,\mu\right), \quad i=0,..,N_t-1,
\end{equation}
\begin{equation}\label{eqn:Full_adjoint_model}
\begin{split}
\quad \quad \textnormal{{Adjoint model}:  } & \mathbf{\lambda}_N = - \left( \frac{\partial r_{\Nt}}{\partial \bx_{\Nt}}\right)^T(\bx_{\Nt},\mu), \\
& {\mathbf \lambda}_i = {\mathbf M}_{i+1,i}^*\mathbf{\lambda}_{i+1} -  \left( \frac{\partial r_{i}}{\partial \bx_{i}}\right)^T(\bx_{i},\mu),\quad i=\Nt-1,..,0,
\end{split}
\end{equation}
\begin{equation}\label{eqn:Full_function_gradient}
\textnormal{{Reduced gradient:  }} \nabla_{{\bf x}_0^*}\mathcal{L} = - {\mathbf \lambda}_0 = 0,
\end{equation}
\end{subequations}
where ${\mathbf M}_{i+1,i}^*$ is the Jacobian matrix of $\M_{i,i+1}$.
\end{theorem}

\begin{proof}
Using the Lagrange multiplier technique the constrained optimization problem \eqref{eqn::Qoi_optimization} is  replaced with the unconstrained
optimization of the following Lagrangian function, $\mathcal{L}: \mathbb{R}^{\Ns} \to \mathbb{R}$
\begin{equation}\label{eqn:Lagrangian_cost_function}
\mathcal{L}({\mathbf x}_0) = \sum_{i=0}^{\Nt}~r_i\left(\x_i,\mu\right) + \sum_{i=0}^{N-1}{\mathbf \lambda}_{i+1}^T\big({\mathbf x}_{i+1}-\M_{i,i+1}
\left({\mathbf x}_i,\mu\right)\big),
\end{equation}
where ${\mathbf \lambda}_i \in \mathbb{R}^{N_\textnormal{state}}$ is the Lagrange multipliers vector at observation time $t_i$.

 An infinitesimal change in $\mathcal{L}$ due to an infinitesimal  change $\delta {\mathbf x}_0$  in  ${\mathbf x}_0$ is
\begin{equation}\label{eqn:Lagrangian_cost_function_variationI}
\begin{split}
 \delta \mathcal{L}({\mathbf x}_0) &= - \sum_{i=0}^{\Nt} \frac{\partial r_i}{\partial \bx_i}(\bx_i,\mu)^T \delta \bx_i + \\
 &\sum_{i=0}^{\Nt-1} \lambda_{i+1}^T\big(\delta {\mathbf x}_{i+1} - {\mathbf M}_{i,i+1}\delta {\mathbf x}_i\big) + \sum_{i=0}^{\Nt-1} \delta
 {\mathbf \lambda}_{i+1}^T\big({\mathbf x}_{i+1} - \M_{i,i+1}\left({\mathbf x}_i,\mu\right)\big), \\
\end{split}
\end{equation}
where $\delta {\mathbf x}_i = \frac{\partial {\mathbf x}_i}{\partial {\mathbf x}_0} \delta {\mathbf x}_0$ and ${\mathbf M}_{i,i+1}$ is the
Jacobian matrix of $\M_{i,i+1}$ with respect to $\x_i$ for all time instances $t_i$, $i=0,\dots,\Nt-1$, at ${\bf x}_i$,
\begin{equation*}\label{eqn:tangent_linear_operators}
 {\mathbf M}_{i,i+1} = \frac{\partial \M_{i,i+1}}{\partial {\mathbf x}_i}({\mathbf x}_i,\mu) \in \mathbb{R}^{N_\textnormal{state} \times N_\textnormal{state}}.
\end{equation*}
The corresponding adjoint operator is ${\bf M}_{i+1,i}^* \in
\mathbb{R}^{N_\textnormal{state} \times N_\textnormal{state}}$ and satisfies
\begin{equation*}\label{eqn:adjoint_operators}
 \langle {\mathbf M}_{i,i+1} z_1, z_2 \rangle_{\mathbb{R}^{N_{\textnormal{state}}}} = \langle z_1, {\mathbf M}^*_{i+1,i} z_2 \rangle_{\mathbb{R}^{N_{\textnormal{state}}}},\quad \forall z_1,~z_2 \in \mathcal {\mathbb{R}^{N_{\textnormal{state}}}},\\
 \end{equation*}
where $\langle \cdot , \cdot \rangle_{\mathbb{R}^{N_{\textnormal{state}}}},$ is the corresponding Euclidian product. This means that the adjoint
operator of ${\mathbf M}_{i,i+1}$ is nothing but its transpose. To underline that the corresponding adjoint model is running backward in time we denote it by
${\bf M}_{i+1,i}^*$.

After rearranging the above equation and using again the definition of the adjoint operators we obtain
\begin{equation}\label{eqn:Lagrangian_cost_function_variationII}
\begin{split}
 &  \delta\mathcal{L}({\mathbf x}_0) = \delta \bx_{\Nt}^T\big( \frac{\partial r_{\Nt}}{\partial \bx_{\Nt}}(\bx_{\Nt}) + {\mathbf \lambda}_{\Nt} \big)+
 \sum_{i=1}^{N-1}\delta {\mathbf x}_i^T \bigg({\mathbf \lambda}_i - {\mathbf M}_{i+1,i}^*\mathbf{\lambda}_{i+1} + \frac{\partial r_{i}}{\partial \bx_{i}}(\bx_{i})\bigg) \\
 & -\delta {\mathbf x}_0^T\bigg({\mathbf M}_{1,0}^*\mathbf{\lambda}_{1} - \frac{\partial r_{0}}{\partial \bx_{0}}(\bx_{0}) \bigg) +
 \sum_{i=0}^{N-1}\delta {\mathbf \lambda}_{i+1}^T \big({\mathbf x}_{i+1} - \M_{i,i+1}\left({\mathbf x}_i,\mu\right)\big).
\end{split}
\end{equation}
By setting to zero the perturbations with respect to $\delta {\mathbf \lambda}_i$ and $\delta \bx_i$, $i=0,\ldots,\Nt$, the first order necessary optimality conditions \eqref{eqn:KKT_Full_forward_model} are obtained.
\end{proof}

While the gradient is equal to zero at the optimum, the formula \eqref{eqn:Full_function_gradient} can still be used to evaluate the gradient of $\mathcal{Q}$ away from the optimum
\begin{equation}
\label{eqn::first_order_approx_gradient}
\nabla_{{\bf x}_0}\mathcal{Q} = \nabla_{{\bf x}_0}\mathcal{L}  = - {\mathbf \lambda}_0.
\end{equation}
Then for ${\bar \bx_0} \in \mathbb{R}^{\Ns}$ in a vicinity of any $\bx_0  \in \mathbb{R}^{\Ns}$, the following first order approximation holds
\begin{equation}
\label{eqn::first_order_approx_adjoint}
\mathcal{Q}({\bar \bx_0},\mu) - \mathcal{Q}({\bx_0},\mu) \approx - {\mathbf \lambda}_0^T\cdot \big({\bar \bx_0} - {\bx_0}\big).
\end{equation}
%

\subsection{Error estimates}\label{sec:Errors_estimates}

We now introduce the first a-posteriori error estimation result.

\begin{theorem}[A-posteriori error estimation]
\label{th::aposteriori_res1}
Let $\x = \{\x_0,~\x_1,\ldots,~\x_{\Nt}\}$ be the solution of the  high-resolution \eqref{eqn:full_forward_model_init} model, and $\widehat{\bx} =\{\widehat{\bx}_0,\widehat{\bx}_1,\ldots,\widehat{\bx}_{\Nt}\}$ the projection of reduced order model solution\eqref{eqn:reduced_forward_model_init} onto the full space. Moreover, let $\x^i = \{{\widehat{\bx}_i},~\x^i_{i+1},\ldots,\x^i_{\Nt}\}$ be the partial trajectories obtained via the full model \eqref{eqn:full_forward_model_init} using as initial conditions the solution of reduced order model \eqref{eqn:reduced_forward_model_init} at time $t_i$ projected onto the full space, i.e. ${\widehat{\bx}_i} = U\, {\tx_i}$ and
\begin{equation}
\label{eqn:partial_trajectory}
 \x^i_{\ell} = \M_{i,\ell}\left({\widehat{\bx}_i},\mu\right), \quad \ell=i+1,\dots,\Nt,
 \quad i=0,\dots,\Nt-1.
\end{equation}
The partial trajectory $\x^i$ contains only $\Nt-i+1$ time steps.

Assume that the reduced order model solution $\widehat{\x}_{i}$ is in a neighborhood of the high-resolution model solution $\x_i,~i=0,\ldots,\Nt$ and if the assumptions of Theorem \ref{th::first_order_nc} hold, then the first order estimation of the error between the quantity of interest computed with the high-fidelity and reduced order models \eqref{eqn::error_qoi} is given by
\begin{equation}
\label{eqn::aposteriori1}
\varepsilon(\mu) \approx - \sum_{i=0}^{Nt} {\widehat{\mathbf \lambda}}_i^T \cdot \Delta \x_i,
\end{equation}
where the model residuals are
\begin{equation}\label{eqn::deltax}
\Delta \x_0 = \x_0 - \widehat{\x}_0, \quad \Delta \x_i = \x^{i-1}_i - \widehat{\x}_i, ~~ i = 1,\dots, N_t,
\end{equation}
and ${\widehat{\mathbf \lambda}}_i,~i=0,\ldots,\Nt-1$,  is the solution of the full adjoint model \eqref{eqn:Full_adjoint_model} linearized about the trajectory  $\x^i,~i=0,\ldots,\Nt-1.$ At the final time step $$\widehat{\mathbf \lambda}_{\Nt} = -(\partial r_{\Nt}/\partial \x_{\Nt})^T(\widehat{\x}_{Nt},\mu).$$
\end{theorem}

\begin{proof}
For each partial trajectory $\x^i$ we introduce the associated quantity of interest
\begin{equation}\label{eqn::partial_qois}
  \mathcal{Q}^i(\widehat{\x}_{i},\mu) = r_i(\widehat{\x}_{i},\mu) + \sum_{j=i+1}^{\Nt}r_j(\x^i_{j},\mu),
  \quad i=0,1,..,\Nt-1.
\end{equation}

Since the trajectories are computed using the high-fidelity model, the gradient of each quantity of interest with respect to $\widehat{\x}_{i}$ is given by a relation analogous to \eqref{eqn::first_order_approx_gradient}. Since the reduced order model solution $\widehat{\x}_{i}$ lies in a neighborhood of the high-resolution model solution $\x_i,~i=0,\ldots,\Nt$, then from equation \eqref{eqn::first_order_approx_adjoint}, we obtain the following first order approximations
\begin{equation}\label{eqn::th_result1}
\begin{split}
   & \mathcal{Q}^0(\x_0,\mu) - \mathcal{Q}^0(\widehat{\bx}_0,\mu) =
   r_0(\x_0,\mu)-r_0(\widehat{\bx}_0,\mu) + \sum_{j=1}^{Nt} r_j(\x_j,\mu) - r_j(\x_j^0,\mu)  \\
   & \qquad \approx - \widehat{\lambda}_0^T\cdot\Delta \x_0,
\end{split}
\end{equation}
\begin{equation}\label{eqn::th_result2}
\begin{split}
& \mathcal{Q}^i(\x_i^{i-1},\mu) - \mathcal{Q}^i(\widehat{\bx}_i,\mu) =
 r_i(\x_i^{i-1},\mu)-r_i(\widehat{\bx}_i,\mu) + \sum_{j=i+1}^{Nt} r_j(\x_j^i,\mu) - r_j(\x_j^i,\mu)  \\
& \qquad \approx - \widehat{\lambda}_i^T\cdot\Delta \x_i,~i=1,..,\Nt-1,
\end{split}
\end{equation}
where $\widehat{\lambda}_i,~i=0,..,\Nt-1$ is the solution of the adjoint model \eqref{eqn:Full_adjoint_model} linearized at the trajectory $\x^i,~i=0,\ldots,\Nt-1.$

Since $r_{\Nt}$ is continuously differentiable the following first order estimation holds
\begin{equation}\label{eqn::th_result3}
r_{\Nt}(\x_{\Nt}^{\Nt-1},\mu)-r_{\Nt}(\widehat{\bx}_{\Nt},\mu) \approx  \frac{\partial r_{\Nt}}{\partial \x_{\Nt}}(\widehat{\x}_{Nt},\mu)\Delta \x_{Nt}.
\end{equation}

According to \eqref{eqn:Full_adjoint_model} we have $\widehat{\lambda}_{\Nt} = - (\partial r_{\Nt}/\partial \x_{\Nt})^T(\widehat{\x}_{Nt},\mu)$ and adding all three equations
\eqref{eqn::th_result1} - \eqref{eqn::th_result3}, we obtain

\begin{equation}\label{eqn:th_result4}
 \varepsilon(\mu) \approx - \sum_{i=0}^{Nt} \widehat{\lambda}_i^T\cdot\Delta \x_i.
\end{equation}
\end{proof}

Figure \ref{fig:Aposteriori_error_estimate} explains the procedure at an intuitive level. The discrepancy in the quantity of interest \eqref{eqn:hfqoi} computed with the high-resolution model solution (bottom trajectory) and projected reduced order model solution (top trajectory) can be estimated by adding all the intermediate dot products.
\begin{figure}[t!]
\centering
\includegraphics[trim=2cm 3.4cm 1cm 6cm ,clip=true,scale=0.6]{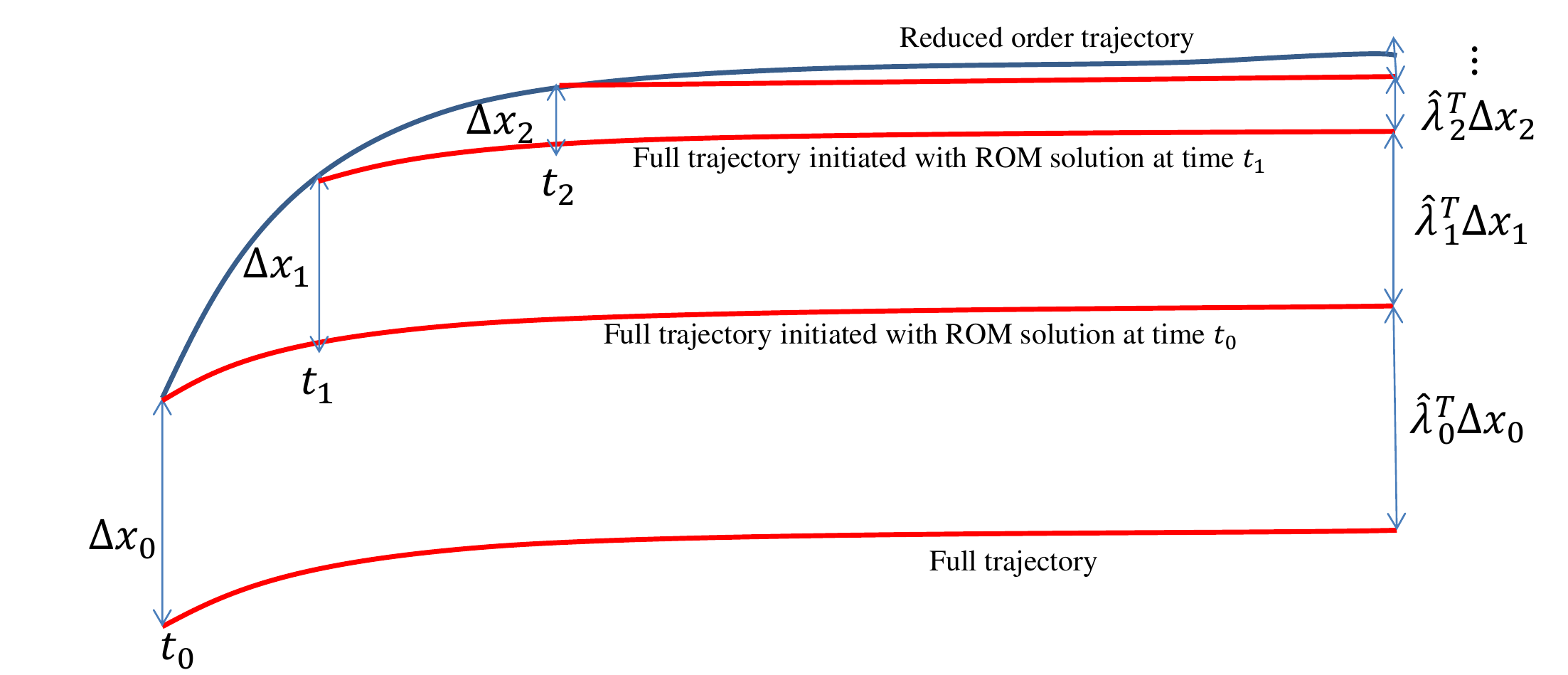}
\caption{Geometrical interpretation of the a-posteriori error estimate \eqref{eqn::aposteriori1}.}
\label{fig:Aposteriori_error_estimate}
\end{figure}
%

\subsection{Efficient computation of the error estimates}
\label{sec:fast_version}

The a-posteriori error estimation result developed in Theorem \ref{th::aposteriori_res1} requires one full, and several partial high-fidelity model runs
along with their associated high-fidelity adjoint model runs. Here we propose efficient a-posteriori estimation results for a general explicit and implicit time integration schemes.

\paragraph{Explicit Euler scheme} The general model introduced in \eqref{eqn:full_forward_model_init} can be described by
\begin{equation}
\label{eqn::explicit_Euler}
  \x_{i+1} = \x_i + h\,{\bf F}(\x_i,\mu),\quad i=0,..,\Nt-1,
\end{equation}
where $h$ is the selected discrete time step.

Using the reduced order solution $\tx_i$ of model \eqref{eqn:reduced_forward_model_init} we perform one time step integration with both high-fidelity and  reduced order models
\begin{eqnarray}
\label{eqn::fast_apost1a}
\x_{i+1}^i& = & U_{\mu}{\tx}_i + h\,\F(U_{\mu}{\tx}_i,{\mu}), \\
\label{eqn::fast_apost1b}
{\tx}_{i+1} & = & {\tx}_i +  h\,U_{\mu}^T\,V_{\mu}\,(P^TV_{\mu})^{-1}\,P^T\,{\F}(U_{\mu}{\tx}_{i},\mu).
\end{eqnarray}
By multiplying \eqref{eqn::fast_apost1b} with $U_\mu$ from the left and subtracting the result from \eqref{eqn::fast_apost1a} we obtain
\begin{eqnarray}
\label{eqn::fast_apost2}
 \Delta \x_{i+1} &=& \x_{i+1}^i - U_\mu\, {\tx}_{i+1} \\
 \nonumber
 &=& h\, \bigl(\mathbf{I}-U_{\mu}U_{\mu}^TV_{\mu}\,(P^TV_{\mu})^{-1}\,P^T\bigr)\, \F(\widehat{\x}_i,{\mu}) \\
 \nonumber
 &=& -\phi(\widehat{\x}_{i},\widehat{\x}_{i+1},\mu),  \quad i=0,..,\Nt-1
\end{eqnarray}
where $\phi: \mathbb{R}^\Ns \times \mathbb{R}^\Ns \times \mathcal{\tilde P}\to \mathbb{R}^\Ns $ is the residual associated with the explicit full model
\begin{equation}\label{eqn::residual_explicit}
\phi_{i+1}(\mu) = \phi(\widehat\x_{i},\widehat\x_{i+1},\mu) =  {\widehat\x}_{i+1} -\widehat\x_i - h\F({\widehat\x}_i,\mu),~i=0,..,\Nt-1.
\end{equation}
\begin{remark}\label{remark_fast_solution}
If the projected reduced order model solutions $\widehat{\x}$ is accurate with respect to the high-fidelity model solution
$\x$ then the partial trajectories $\x^i = \{{\widehat{\bx}_i},~\x^i_{i+1},..$, $\x^i_{\Nt}\}$, $i=0,..,\Nt-1$ can be approximated by
truncated trajectories obtained using one single high-fidelity model run $\{\x_0,~\x_1,..,~\x_{\Nt}\}$. Then for estimating
${\widehat{\mathbf \lambda}}_i^T,~i=0,\ldots,\Nt$ in \eqref{eqn::aposteriori1} we need only a single high-fidelity adjoint model run
instead of several partial high-fidelity adjoint model trajectories required by Theorem \ref{th::aposteriori_res1}.
\end{remark}
\begin{remark}\label{remark_fast_solutionb}
{Unlike the Galerkin POD residual, the DEIM based residual \eqref{eqn::residual_explicit} is not orthogonal to the reduced manifold $U_{\mu}$. Thus we can compute the mismatch in \eqref{eqn::error_qoi} by employing a reduced order adjoint model solution

\begin{equation}\label{eqn::fast_apost3}
  \varepsilon \approx - \big[ U_{\mu} \mathbf{\tilde  \lambda}_0\big]^T(\x_0 - \widehat{\x}_0) + \sum_{i=1}^{\Nt}\big[ U_{\mu} \mathbf{\tilde  \lambda}_{i}\big]^T\phi_{i}(\mu).
\end{equation}
By incorporating the high-fidelity adjoint snapshots for the construction of the reduced order basis \cite{cstefuanescu2015pod} we can design an accurate reduced order adjoint model. The error estimate proposed in \eqref{eqn::fast_apost3} requires only one reduced forward and one adjoint model runs as well as evaluating the residuals in \eqref{eqn::residual_explicit}}. Here $\mathbf{\tilde  \lambda}_{i} \in \mathbb{R}^k$ denotes the solution of the reduced adjoint model at time $t_i$
\begin{equation}\label{eqn::Explicit_reduced_adjoint_model}
\begin{split}
& \mathbf{\tilde \lambda}_{\Nt} = -U_{\mu}^T\frac{\partial r_{\Nt}}{\partial \bx_{\Nt}}(U_{\mu}\tx_{\Nt},\mu), \\
& \mathbf{\tilde \lambda}_i = U_{\mu}^T\big[I + hV_{\mu}\,(P^TV_{\mu})^{-1}\,P^T\frac{\partial \F}{\partial \x_i}(U_{\mu}\tx_i,\mu)\big]^TU_{\mu}\mathbf{\tilde \lambda}_{i+1}- U_{\mu}^T\frac{\partial r_{i}}{\partial \bx_{i}}(U_{\mu}\tx_{i},\mu), \\
& \quad \quad \quad \quad \quad \quad \quad \quad \quad \quad \quad \quad \quad \quad \quad  \quad \quad \quad \quad \quad \quad i=\Nt-1,..,0. \\
\end{split}
\end{equation}
\end{remark}
Techniques for fast evaluations of the reduced Jacobians exist in the literature and a comparison between them is available in \cite{cstefuanescu2017efficient}.

\begin{remark}\label{remark_global_model}
 The use of a reduced adjoint model \eqref{eqn::Explicit_reduced_adjoint_model} allows the a-posteriori error estimation result in Theorem
 \ref{th::aposteriori_res1} to be exploited for different parametric configuration. The more accurate the global reduced forward and adjoint models are, the more precisely
 the a-posteriori estimate in Theorem \ref{th::aposteriori_res1} is.
\end{remark}

\paragraph{Implicit Euler scheme} The model \eqref{eqn:full_forward_model_init} has the following form
\begin{equation}\label{eqn::implicit_Euler}
  \x_{i+1} = \x_i + h\,{\bf F}(\x_{i+1},{\mu}),\quad i=0,..,\Nt-1,
\end{equation}
where each component of $\F: \mathbb{R}^{\Ns \times \mathcal{\tilde P}} \to \mathbb{R}^{\Ns}$ is assumed to be of class $C^1$ in the first variable. We take one implicit Euler step with the full and reduced order models initialized at the reduced order solution:
\begin{eqnarray}
\label{eqn::fast_apost_implicit1}
\x_{i+1}^i& = & U_{\mu}\, {\tx}_i + h\, \F(\x_{i+1}^i,{\mu}), \\
 \label{eqn::fast_apost_implicit2}
{\tx}_{i+1} & = & {\tx}_i +  h\, U_{\mu}^T\,V_{\mu}\,(P^TV_{\mu})^{-1}\,P^T{\F}(U_{\mu}{\tx}_{i+1},{\mu}). \end{eqnarray}
By subtracting \eqref{eqn::fast_apost_implicit2} from \eqref{eqn::fast_apost_implicit1} after projecting equation \eqref{eqn::fast_apost_implicit2} to the full space we obtain
\begin{equation}
\label{eqn::fast_apost_implicit3}
\begin{split}
& \Delta \x_{i+1} = \x_{i+1}^i - U_\mu\, {\tx}_{i+1} \\
&= \F(\x_{i+1}^i,{\mu})-h\, U_{\mu}U_{\mu}^T\, V_{\mu}\,(P^TV_{\mu})^{-1}\,P^T{\F}(\widehat{\x}_{i+1},{\mu}) \\
&= \big[\F(\x_{i+1}^i,{\mu}) - \F(\widehat{\x}_{i+1},{\mu}) +
\F(\widehat{\x}_{i+1},{\mu}) \\
& \qquad - \ U_{\mu}U_{\mu}^TV_{\mu}\,(P^TV_{\mu})^{-1}\,P^T\F(\widehat{\x}_{i+1},{\mu})\big], \qquad i=0,\ldots,\Nt-1.
\end{split}
\end{equation}

The error \eqref{eqn::fast_apost_implicit3} is approximated to first order by
\begin{equation}
\label{eqn::fast_apost_implicit4}
\Delta \x_{i+1} \approx h\, \left[\frac{\partial \F}{\partial \x_{i+1}}(\widehat{\x}_{i+1},\mu) \Delta \x_{i+1} +(\mathbf{I} - U_{\mu}U_{\mu}^TV_{\mu}\,(P^TV_{\mu})^{-1}\,P^T)\F(\widehat{\x}_{i+1}) \right]
\end{equation}
and then
\begin{equation}
\label{eqn::fast_apost_implicit5}
 \Delta \x_{i+1}  \approx h\, \left[\mathbf{I} - h\frac{\partial \F}{\partial \x_{i+1}}(\widehat{\x}_{i+1},\mu) \right]^{-1}\, \left(\mathbf{I} - U_{\mu}U_{\mu}^TV_{\mu}\,(P^TV_{\mu})^{-1}\,P^T)\F(\widehat{\x}_{i+1}\right).
\end{equation}
The high-fidelity adjoint model solution $\mathbf{\lambda}_i$ of the implicit Euler model associated with the  quantity of interest defined in \eqref{eqn:hfqoi} is
\begin{equation}
\begin{split}
\label{eqn::adjoint_implicit_Euler}
& \lambda_{\Nt} = -\left(\frac{\partial r_{\Nt}}{ \partial \x_{\Nt}}\right)^T(\x_{\Nt},\mu) \,,  \\
& {\lambda}_i =   \left[\mathbf{I} - h\frac{\partial \F}{\partial \x_{i+1}}(\widehat{\x}_{i+1},\mu) \right]^{-T}\, \lambda_{i+1} - \left(\frac{\partial r_{i}}{ \partial \x_{i}}\right)^T(\x_{i},\mu) ,\quad i=\Nt-1,\dots , 0\,.
\end{split}
\end{equation}

As stated in Remarks \ref{remark_fast_solution} and \ref{remark_fast_solutionb}, we can assume  accurate forward and reduced order model solutions;i.e., $\lambda_{i+1} \approx {\widehat{\mathbf \lambda}}_{i+1}$. Left-multiplying equation \eqref{eqn::fast_apost_implicit5} by ${\widehat{\mathbf \lambda}}_{i+1}$ leads to
\begin{equation}
\label{eqn::fast_apost_implicit6}
\begin{split}
&\widehat{\mathbf{\lambda}}_{i+1}^T\, \Delta \x_{i+1} \\
&\approx h\,\lambda_{i+1}^T\,\left[\mathbf{I} - h\frac{\partial \F}{\partial \x_{i+1}}(\widehat{\x}_{i+1},\mu) \right]^{-1}\,  \left(\mathbf{I} - U_{\mu}U_{\mu}^TV_{\mu}\,(P^TV_{\mu})^{-1}\,P^T\right)\, \F(\widehat{\x}_{i+1},\mu) \\
& =  \left[\lambda_{i} + \frac{\partial r_{i}}{ \partial \x_{i}}(\x_{i},\mu)\right]^T \, \left(\mathbf{I} - U_{\mu}U_{\mu}^TV_{\mu}\,(P^TV_{\mu})^{-1}\,P^T\right)\,\F(\widehat{\x}_{i+1},\mu),\\
&\qquad i=0,\ldots,\Nt-1,
\end{split}
\end{equation}
where the second equality follows from the adjoint model \eqref{eqn::adjoint_implicit_Euler}.

The error in the quantity of interest \eqref{eqn::aposteriori1}  due to the usage of the reduced order model can be also estimated by
\begin{equation}\label{eqn::fast_apost_implicit8}
\varepsilon(\mu) \approx -\big[U_{\mu}\mathbf{\tilde \lambda}_0\big]^T(\x_0 - \widehat{\x}_0) + \sum_{i=0}^{\Nt-1} \phi_{i+1}^T \big[U_{\mu}\mathbf{\tilde
\lambda}_i+\frac{\partial r_i}{\partial \x_i}(U_{\mu}\tx_i,\mu)\big],
\end{equation}
where $\phi: \mathbb{R}^\Ns \times \mathbb{R}^\Ns \times \mathcal{\tilde P} \to \mathbb{R}^\Ns $ is now the residual associated with the implicit full model.
\begin{equation}\label{eqn::residual_implicit}
\phi_{i+1}(\mu) = \phi(\widehat\x_{i},\widehat\x_{i+1},\mu) =  \widehat{\x}_{i+1} -\widehat\x_i - h\F({\widehat\x}_{i+1},\mu),~i=0,..,\Nt-1.
\end{equation}

As in the case of the explicit model, we can compute the estimated error in \eqref{eqn::fast_apost_implicit8} using only one reduced forward and adjoint model runs
and the evaluation of the high-fidelity model residuals \eqref{eqn::residual_implicit}. The statement in Remark \ref{remark_global_model} is valid here too.

\section{Adaptive location of DEIM interpolation points}
\label{sec:DEIM_adapt}

Here we describe a novel algorithm for selection of the DEIM points such that the accuracy of the quantity of interest $\mathcal{Q}$  \eqref{eqn:hfqoi} evaluated using a reduced order model is increased. In \cite{peherstorferonline} the adaptivity mechanism changes the non-linear term reduced basis via rank-one updates. Here we do not change the basis, but rather we adaptively relocate the DEIM interpolation points.

Our adaptive strategy uses the a-posteriori error estimation result  \eqref{eqn::aposteriori1} and the fast computation techniques presented in Section \ref{sec:fast_version}. The individual contribution at each spatial location and time step to the error in the quantity of interest can be calculated by using the Hadamard product $\odot$ instead of the scalar products in \eqref{eqn::fast_apost3} and \eqref{eqn::fast_apost_implicit8}. The Hadamard products are the dual weighted residuals.

For the explicit case the dual weighted residuals are defined as
\begin{equation}\label{eqn::dw_residuals_explicit}
   z_0 = [ U \mathbf{\tilde  \lambda}_0\big]\odot(\x_0 - \widehat{\x}_0); \quad z_i = [ U \mathbf{\tilde  \lambda}_{i}\big]\odot \phi_{i},~i=1,..,\Nt,
\end{equation}
while for the implicit case they are defined as
\begin{equation}\label{eqn::dw_residuals_implicit}
   z_0 = [ U \mathbf{\tilde  \lambda}_0\big]\odot(\x_0 - \widehat{\x}_0); \quad z_i = \phi_i\odot[ U \mathbf{\tilde  \lambda}_{i-1} + \frac{\partial r_{i-1}}{\partial \x_{i-1}}(U\tx_{i-1},\mu)\big],~i=1,..,\Nt.
\end{equation}
The high-fidelity explicit and implicit models residuals $\phi_i(\mu)$ are defined in \eqref{eqn::residual_explicit} and \eqref{eqn::residual_implicit}.

Next the dual weighted residuals are collected into a matrix $Z$ and a singular vector decomposition is applied to extract the left singular vectors denoted by $W = \{w_0,~w_1,\ldots,w_m\}$.
Now we are ready to introduce our adaptive strategy. In addition to the non-linear basis $V_{\mu}$, the dual residual basis $W$ is used as input.

\begin{algorithm}
{\bf INPUT}: $\{v_\ell\}_{\ell=1}^m\subset\mathbb{R}^\Ns$ (linearly independent),~$\{w_\ell\}_{\ell=1}^m\subset\mathbb{R}^\Ns$ (linearly independent), $\alpha \in [0,1]$: \newline
{\bf OUTPUT}: $\bm{\rho}=[\rho_1,..,\rho_m]\in\mathbb{N}^m$
 \begin{algorithmic}[1]
 \State $\{\psi^v, \rho_1^v\}=\max|v_1|; ~~{\psi^1}\in \mathbb{R}$ is the largest absolute value among entries of $v_1$, and $\rho_1^1$ is its position (the smallest index taken in case of a tie).
 \State $\{\psi^w, \rho_1^w\}=\max|w_1|,~~{\psi^2}\in \mathbb{R}$.
 \State Set $\rho_1=\rho_1^v$ if $\psi^v \ge \psi^w$, or $\rho_1=\rho_1^w$ otherwise.
 \State $V_{\mu}:=[v_1]\in \mathbb{R}^n,~~P:=[e_{\rho_1}]\in \mathbb{R}^n,~~\bm{\rho}:=[\rho_1]\in \mathbb{N}.$
 \For{$\ell=2,..,m$}

      \State Solve $(P^TV_\mu)\, {c^v}=P^T\, v_\ell \textrm{  for } {c^v}\in \mathbb{R}^{\ell-1};~V_{\mu},P\in\mathbb{R}^{\Ns\times(\ell-1)}.$
      \State $r^v := v_\ell - V_{\mu}{c^v},~~r^v\in \mathbb{R}^{\Ns}.$
       \State Solve $(P^TV_{\mu})\,{ c^w}=P^T\,w_\ell \textrm{  for } {c^w}\in \mathbb{R}^{\ell-1};~V_{\mu},P\in\mathbb{R}^{\Ns\times(\ell-1)}.$
      \State $r^w := w_\ell-V_{\mu}\,{c^w},~~r^w\in \mathbb{R}^{\Ns}.$
      \State $\{\psi,\rho_\ell\}=\max\,\{\alpha|r^v|+(1-\alpha)|r^w|\}.$
      \State $V_{\mu} := [V_{\mu}~~v_\ell],~~P := [P~~e_{\rho_\ell}],~~\bm{\rho} := \left[
           \bm{\rho}^T ~~\rho_\ell\right]^T.$
   \EndFor
 \end{algorithmic}
 \caption{Adaptive location of DEIM interpolation points}
 \label{alg::DEIM_adaptiv}
\end{algorithm}

The residual $r^w$ estimates the error of the dual weighted residuals representations in the non-linear term subspace $V_{\mu}$. In contrast with the traditional DEIM approach, we place the points where the highest value of the combined residual $\alpha|r^v|+(1-\alpha)|r^w|$ is found. This allows to generate a satisfactory globally accurate non-linear reduced order term, while enhancing the accuracy in the spatial locations with the higher contributions to \eqref{eqn::error_qoi}. Here the parameter $\alpha$ is chosen heuristically, and finding an automated selection procedure is an open problem.

\section{Numerical experiments}\label{sec:numerical}
We evaluate the a-posteriori error estimation results and adaptive DEIM algorithm for two non-linear test problems, the one-dimensional Burgers and the two-dimensional
Shallow Water Equations. Proper orthogonal decomposition and discrete empirical interpolation methods are applied along with the Galerkin projection to generate the associated. We validate the a-posteriori error formulas \eqref{eqn::aposteriori1} and \eqref{eqn::fast_apost_implicit8} and show that the adaptive DEIM strategy proposed in Section \ref{sec:DEIM_adapt}
is successfully in reducing the error \eqref{eqn::error_qoi}. Our experiments are performed for the same parametric configurations for the Swallow Water Equations model. Additional efforts for the 1D-Burgers model confirm the
Remark \ref{remark_global_model} statement, thus enabling the use of the a-posteriori error results for parametric
configurations others than the one used to construct the reduced order bases for both state variables and non-linear terms.

\subsection{One-dimensional Burgers model}\label{subsec:Burgers}

Burgers' equation \cite{burgers1974nonlinear,burgers2013nonlinear} provides a simplification of the equations of fluid dynamics by omitting the pressure terms.
For a given viscosity coefficient $\mu$, the evolution of the fluid velocity $u$ is given by
\begin{equation}
\frac{\partial u}{\partial t} + u\, \frac{\partial u}{\partial x} = \mu\, \frac{\partial^2 u}{\partial x^2}, \quad x \in [0,L], \quad t \in (0,t_\textnormal{f}].
\label{eqn:Burgers-pde}
\end{equation}
We assume Dirichlet homogeneous boundary conditions $u(0,t) = u(L,t) = 0,~t \in (0,t_\textnormal{f}]$ and as initial conditions we use a seventh degree polynomial.

The discretization uses a uniform spatial mesh with $n$ grid points and $\Delta x=L/(n-1)$, and uniform temporal mesh with $N_t$ grid points and $\Delta t=t_\textnormal{f}/(N_t-1)$. The discrete solution vector is
denoted by ${\boldsymbol u}(t_{N})\approx [u(x_i,t_{N})] \in \mathbb{R}^{\Ns},~N=1,2,\ldots,N_t$ (the spatial dimension is $\Ns = n-2$ after eliminating the known
boundaries). The semi-discrete version of the 1D Burgers model \eqref{eqn:Burgers-pde} reads
\begin{equation}
\label{eqn:Burgers-sd}
 {\bf u}'  =  -{\bf u}\odot A_x{\boldsymbol u} + \mu A_{xx}{\boldsymbol u},
\end{equation}
where ${\bf u}'$ denotes the time derivative of ${\bf u}$. $A_x$, $A_{xx}\in \mathbb{R}^{\Ns\times \Ns}$ are the central difference first-order and second-order
space derivatives operators which also account for the boundary conditions, respectively.

The implicit Euler method is employed for time discretization and it is implemented in Matlab.
The non-linear algebraic systems are solved using Newton-Raphson method and the maximum number of Newton iterations allowed each time step is set to $50$.
The solution is considered accurate enough when the euclidian norm of the residual is less than $10^{-10}$.

\subsection{Numerical experiments with the one-dimensional Burgers model}\label{sec:numerical_Burgers}

We propose the following configuration $L~=~1$, $~t_f~=~1$, $n~=~201$, $\Ns~=~199,$ $\Nt~=~201$. The viscosity parameter is set initially to $\mu = 0.1$, and the
initial conditions are depicted in Figure \ref{fig::Burgers_solution}a. The quantity of interest depends only on some particular components of the solution at the final
time step (colored in green in Figure \ref{fig::Burgers_solution}b) and it is defined below
\begin{equation}\label{eqn:Burgers_qoi}
 \mathcal{Q}(u) = \sum_{i=2}^{21} u(x_i,t_{N_t})^2,~[x_2,x_{21}] = [0.05,0.1].
\end{equation}
\begin{figure}[t!]
  \centering
  \subfigure[Initial Conditions] {\includegraphics[scale=0.3]{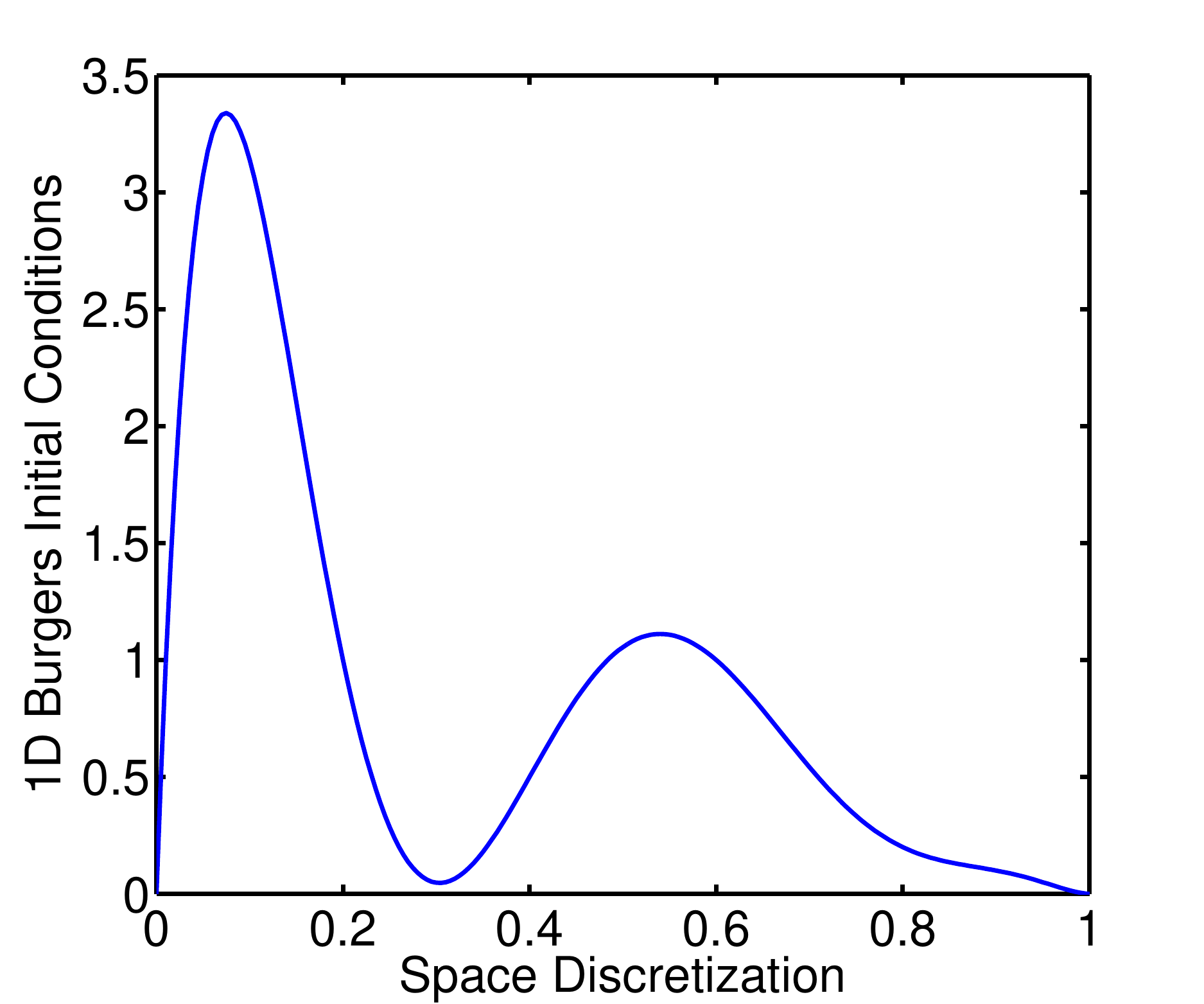}}
  \subfigure[Final time solutions ]{\includegraphics[scale=0.3]{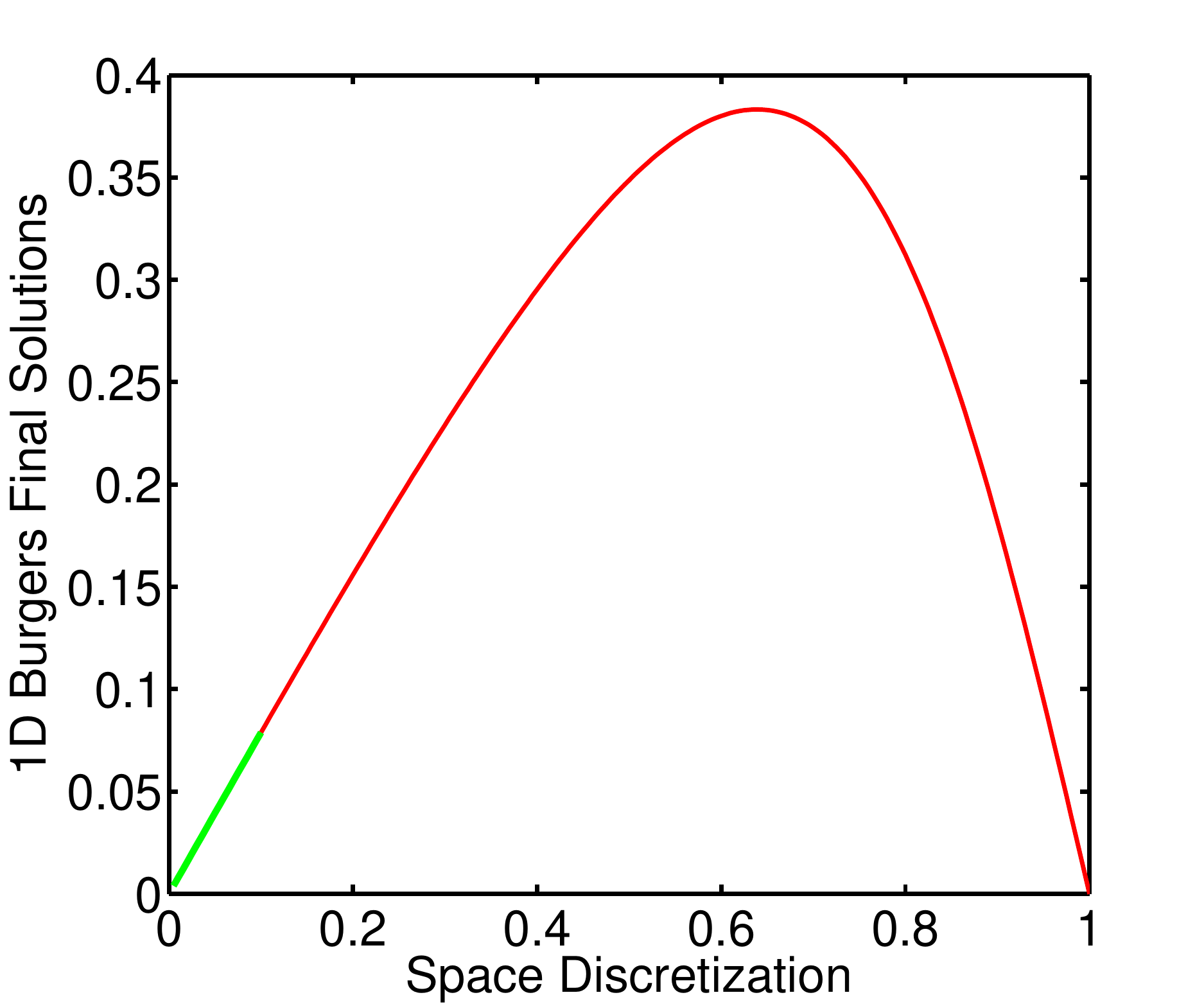}}
\caption{\label{fig::Burgers_solution}Initial conditions of Burgers model for $\mu = 0.1$ (left panel). The Burgers model solution at final time (right panel). The green color depicts the model trajectory used to compute the quantity of interest.}
\end{figure}

A number of $401$ snapshots are used to construct the reduced order basis for the state variable and include the solutions of both high-fidelity forward and adjoint
models. For the advection non-linear term we applied the singular value decomposition and generate the reduced order basis required by the DEIM approximation using
$201$ snapshots. The spectra of the snapshots matrices are illustrated in Figure \ref{fig::singular_values_Burgers}, and for POD basis dimensions larger than $25$,
the a-priori estimate suggests reduced order solutions errors smaller than $1e-5$ for the same parametric configuration. However this estimate does not include integration error and the overall error is usually underestimated.

\begin{figure}[t!]
\centering
\includegraphics[scale=0.34]{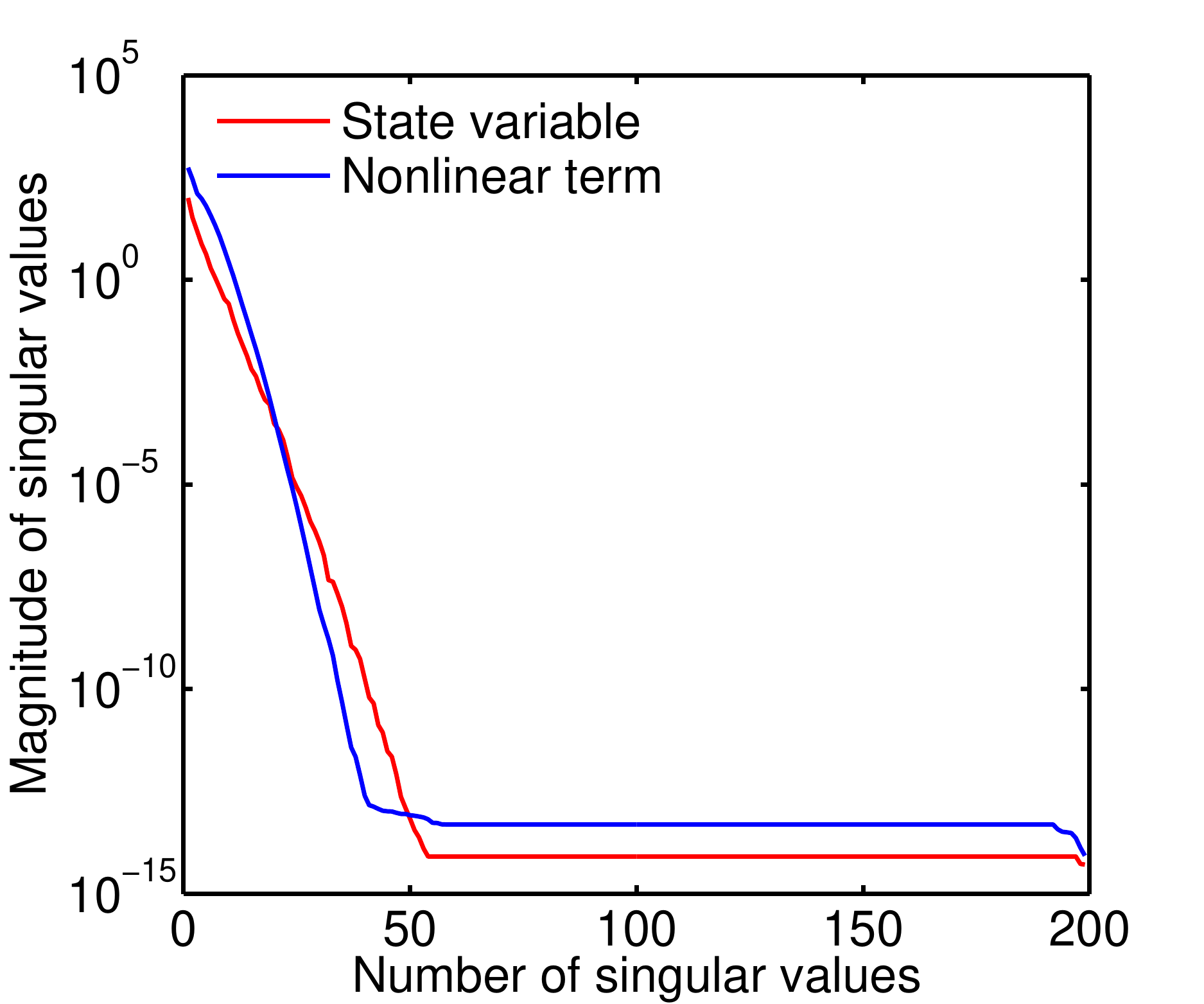}
\caption{\label{fig::singular_values_Burgers} Singular values of state variable and non-linear term of the 1D-Burgers model.}
\end{figure}

Since the implicit Euler method was used to discretize the 1D-Burgers model, we will verify the theoretical a-posteriori error estimate result described in \eqref{eqn::fast_apost_implicit8}. First we compute the error induced by estimating the quantity of interest \eqref{eqn:Burgers_qoi} using the POD/DEIM reduced order model and compare it against the dual weighted residuals sum. The formula in \eqref{eqn::fast_apost_implicit8} requires only reduced order model runs and the results are depicted in Figure \ref{fig::apost_same_config_Burgers}. By increasing the dimension of the reduced manifold, the error estimation gets closer to the true error as seen in Figure \ref{fig::apost_same_config_Burgers}(a) and for POD basis larger than $12$ the discrepancies are lower than $1e-03$. The more accurate the underlying reduced order model the more precise the proposed a-posteriori error estimation result is. The sensitivity of the estimated error \eqref{eqn::fast_apost_implicit8} with respect to the number of DEIM interpolation points is shown in \ref{fig::apost_same_config_Burgers}(b) where dimension of POD basis is set to $15$. For more than $14$ DEIM points the mismatches between the true and estimated errors are smaller then $1e-4$.

\begin{figure}[t!]
  \centering
  \subfigure[Number of DEIM points $=40$] {\includegraphics[scale=0.3]{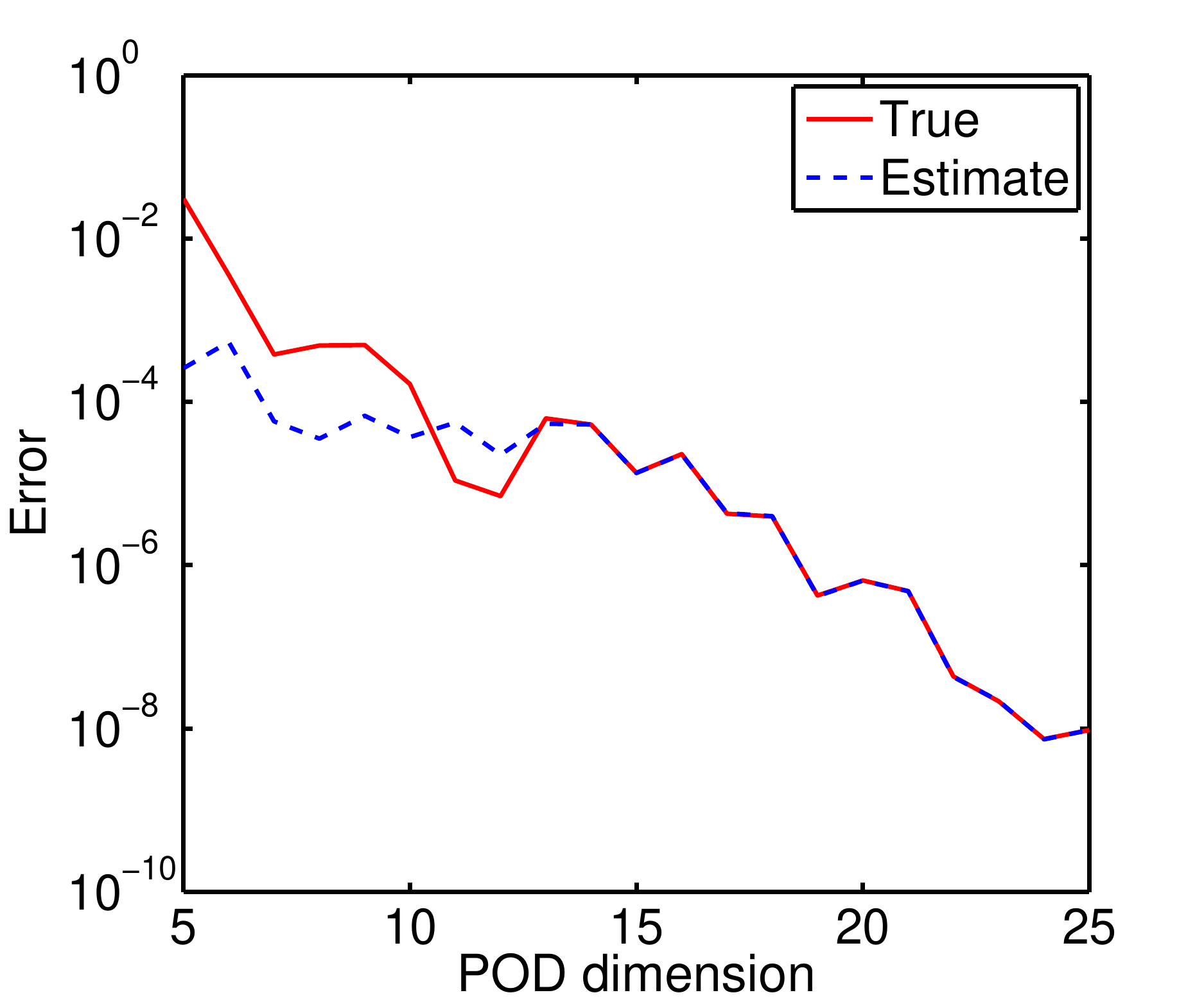}}
  \subfigure[Dimension of POD basis $=15$ ]{\includegraphics[scale=0.3]{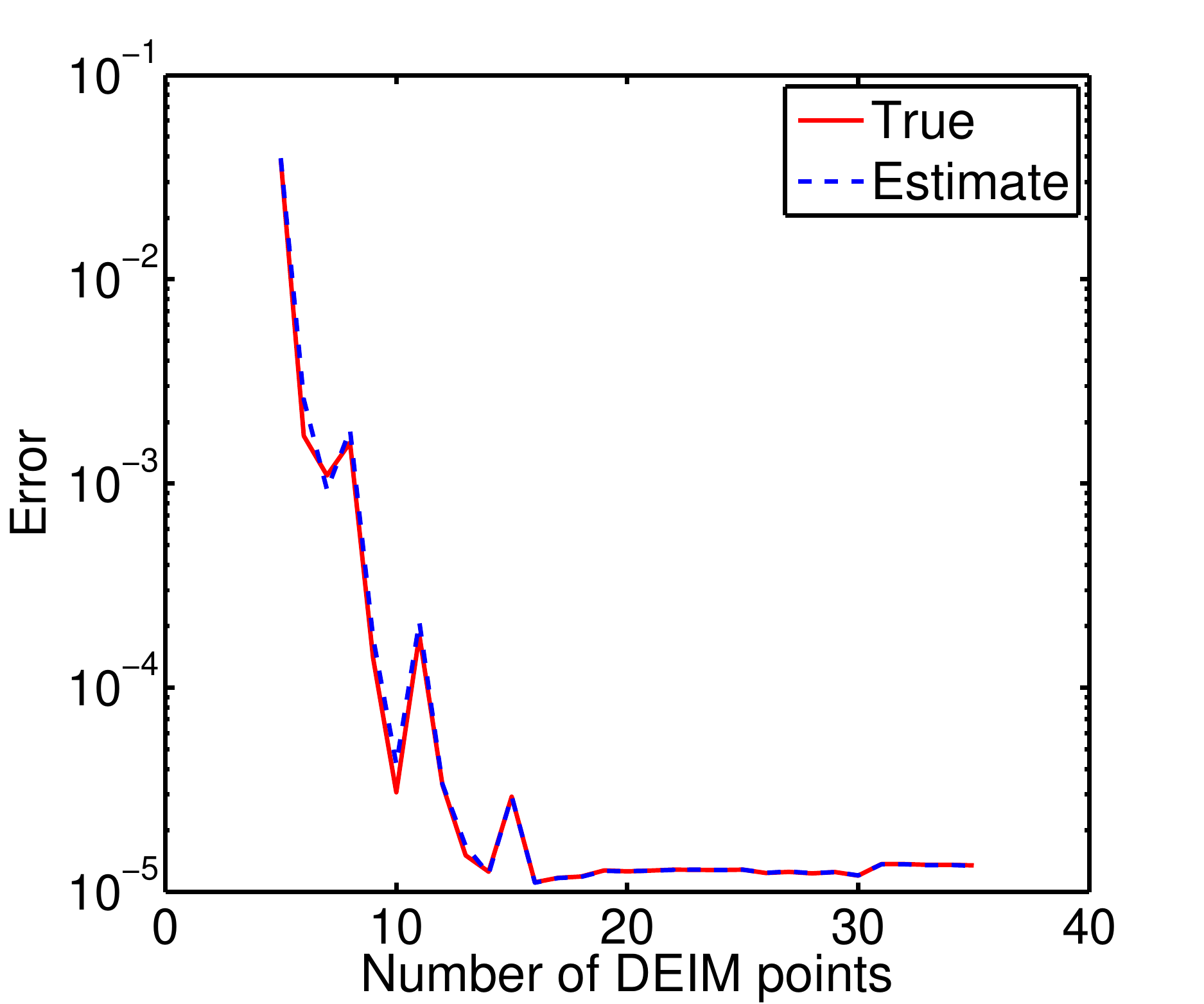}}
\caption{\label{fig::apost_same_config_Burgers}A-posteriori error estimates for the same parametric configuration - $\mu = 0.1$.}
\end{figure}

\begin{figure}[t!]
  \centering
  \subfigure[Number of DEIM points $=40$] {\includegraphics[scale=0.3]{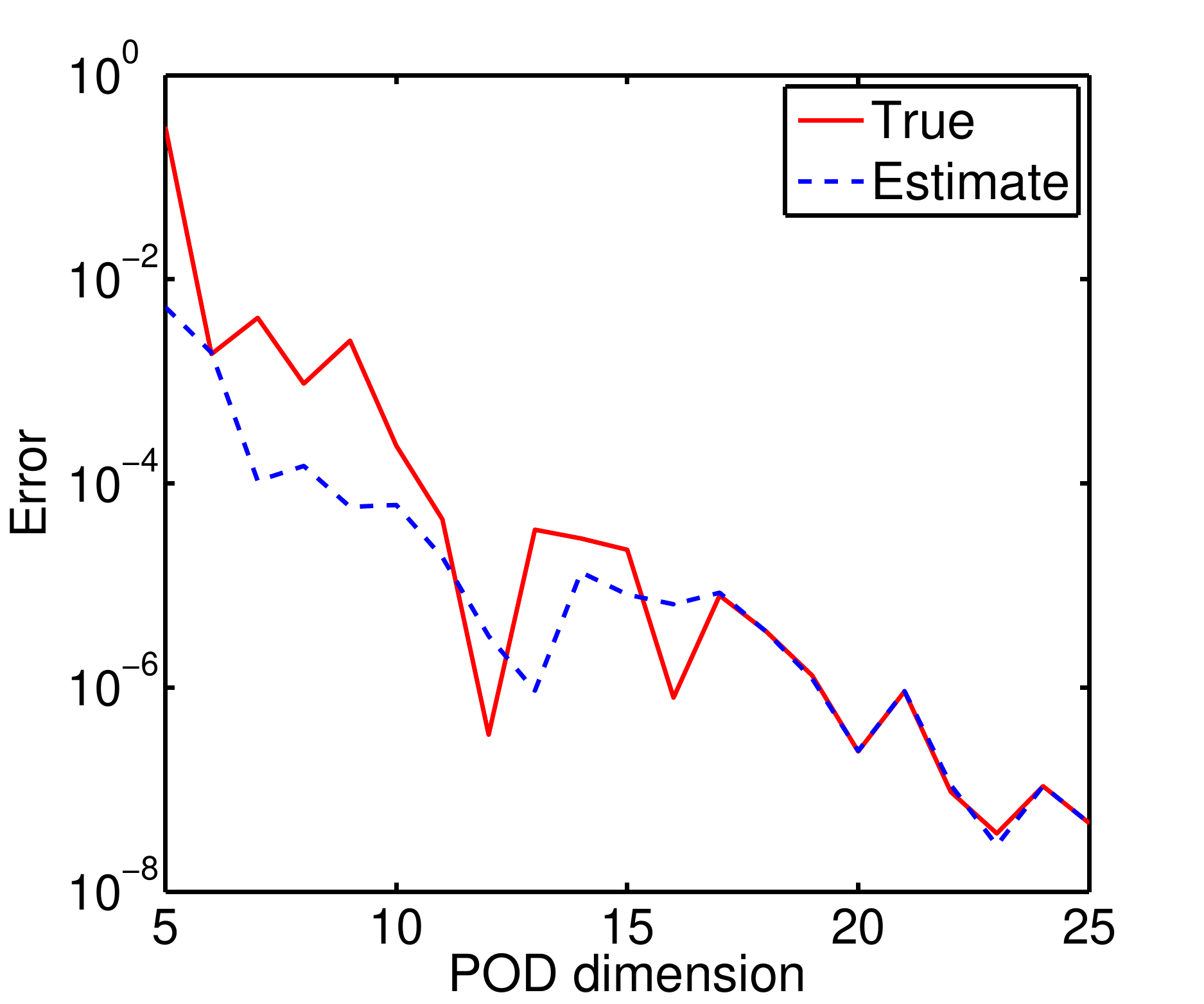}}
  \subfigure[Dimension of POD basis $=15$ ]{\includegraphics[scale=0.3]{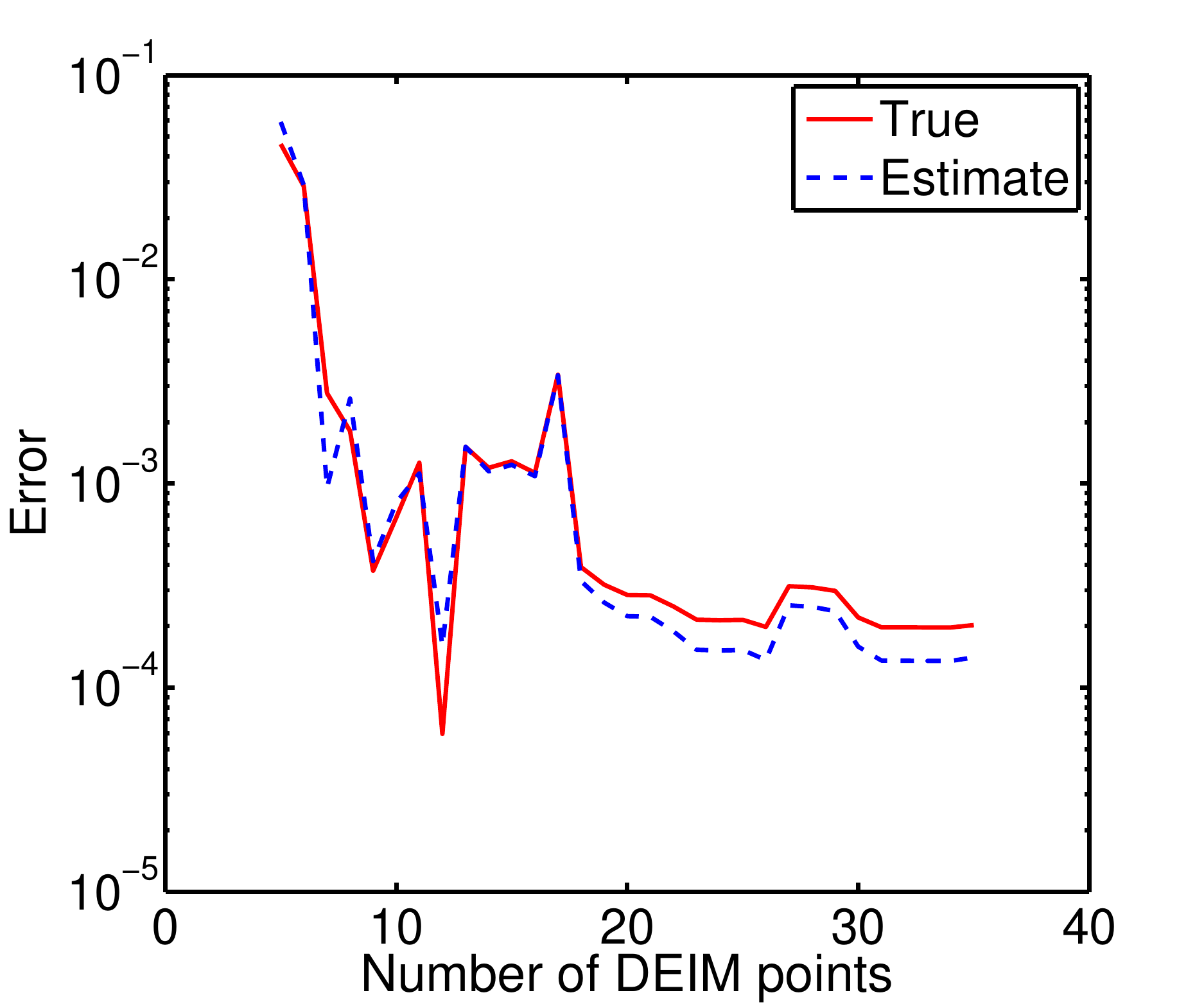}}
\caption{\label{fig::apost_diff_config_Burgers}A-posteriori error estimates for different parametric configuration - $\mu = 0.07$.}
\end{figure}

Next we test the a-posteriori estimation result \eqref{eqn::fast_apost_implicit8} for a different parametric configuration. As such, using bases computed for $\mu=0.1$, we aim to estimate the quantity of interest error for $\mu=0.07$. The approximation will be accurate as long the reduced order forward and adjoint models will be accurate with respect to the high fidelity models. First we set the number of DEIM points to $40$ and compute the estimates for various POD basis dimensions. A comparison against the true error is depicted in Figure \ref{fig::apost_diff_config_Burgers}(a). We notice larger discrepancies than in the case of using similar parametric configurations. For POD basis dimension larger than $17$, our estimation is very accurate.

By fixing the POD basis dimension to $15$ and varying the number of DEIM points, we set up another experiment and the results are shown in Figure \ref{fig::apost_diff_config_Burgers}(b). The a-posteriori estimation is accurate though is less precise than in the case of similar parametric configuration.

\begin{figure}[t!]
\centering
\includegraphics[scale=0.37]{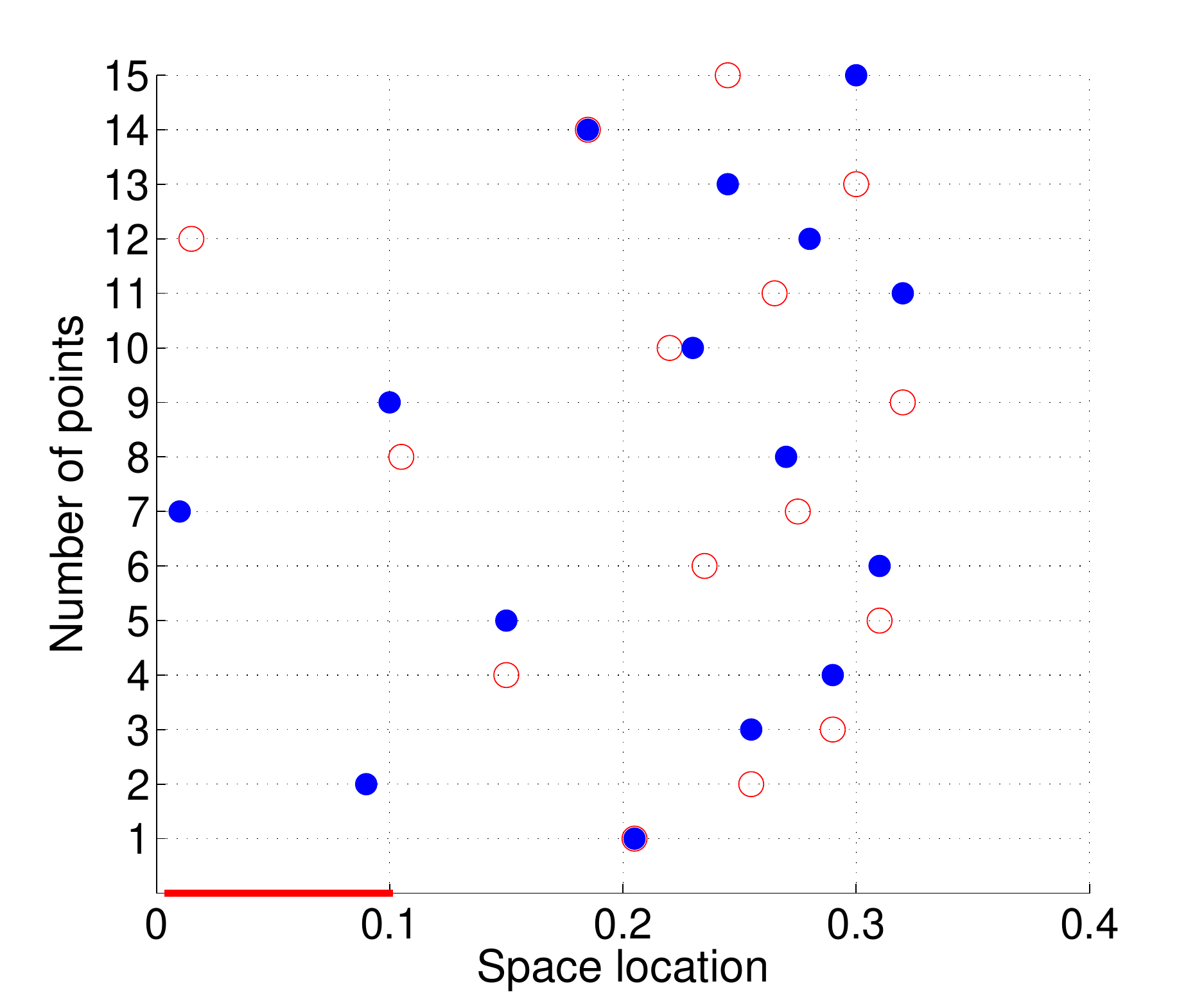}
\caption{\label{fig::DEIM_points_location} DEIM points locations. Adaptive points are depicted in blue and standard points in red.}
\end{figure}
Next we will prove the usefulness of the adaptive DEIM algorithm \ref{alg::DEIM_adaptiv} by enhancing the accuracy of the quantity of interest computed using the updated reduced order model. First we compute the dual weighted residuals \eqref{eqn::dw_residuals_implicit} and perform a singular value decomposition applied to generate left singular vectors. The first $15$ of the singular vectors are stored in matrix $W$ and then used, together with the non-linear term basis $V_{\mu}$, to initiate Algorithm \ref{alg::DEIM_adaptiv}. The parameter $\alpha$ is set to $0.5$

The locations of the DEIM points computed using Algorithms \ref{alg::DEIM} and \ref{alg::DEIM_adaptiv} are shown in Figure \ref{fig::DEIM_points_location}. The adaptive algorithm places more points in the spatial domain $[0.05,0.1]$ of the quantity of interest marked with red in Figure \ref{fig::DEIM_points_location}. The locations of the adaptive DEIM points change the matrix $P$ in DEIM approximation of the non-linear term \eqref{eqn::POD_DEIM_nonlinearity}.


\begin{figure}[t!]
  \centering
  \subfigure[Non-linear approximation] {\includegraphics[scale=0.3]{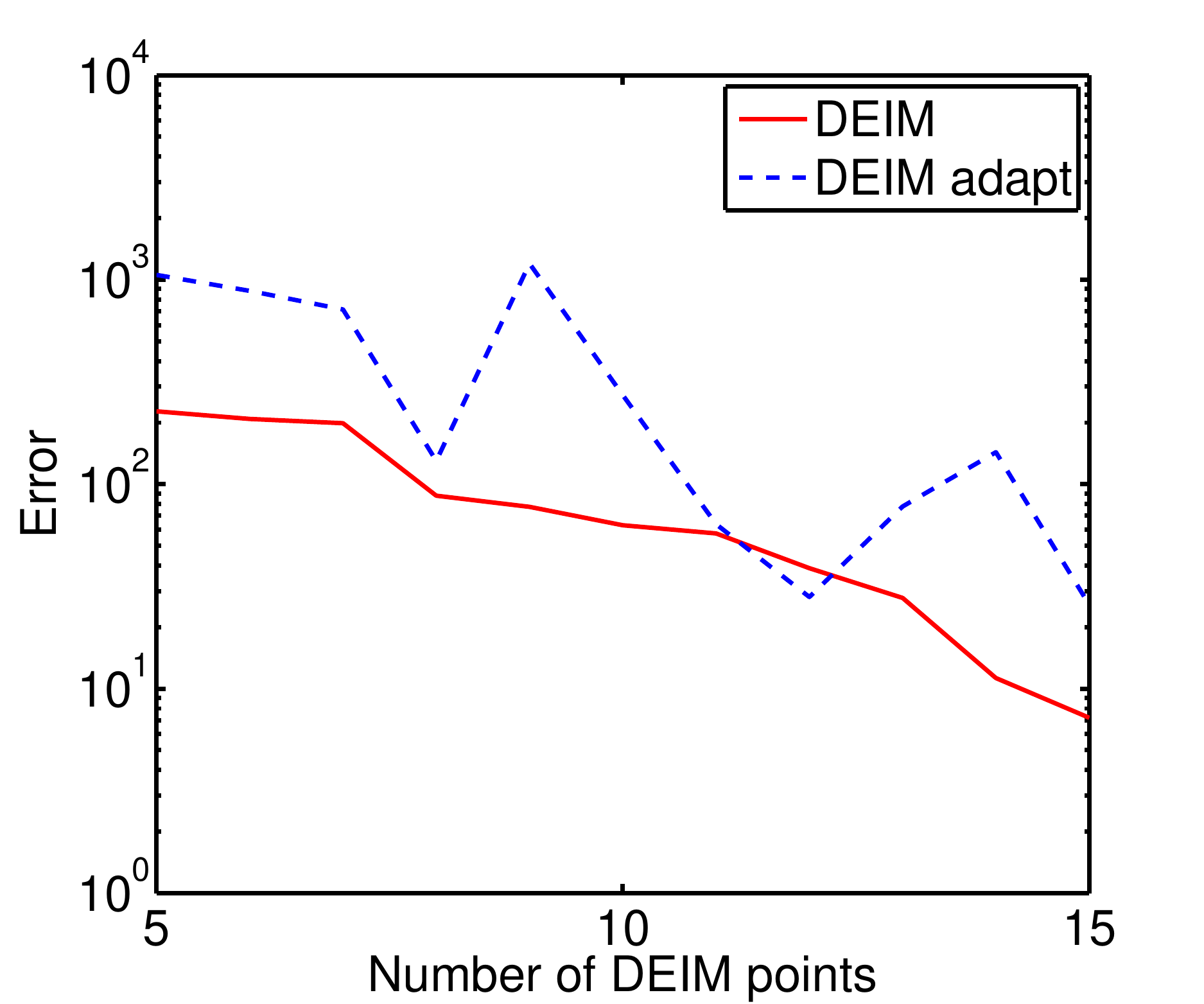}}
  \subfigure[Condition numbers]{\includegraphics[scale=0.3]{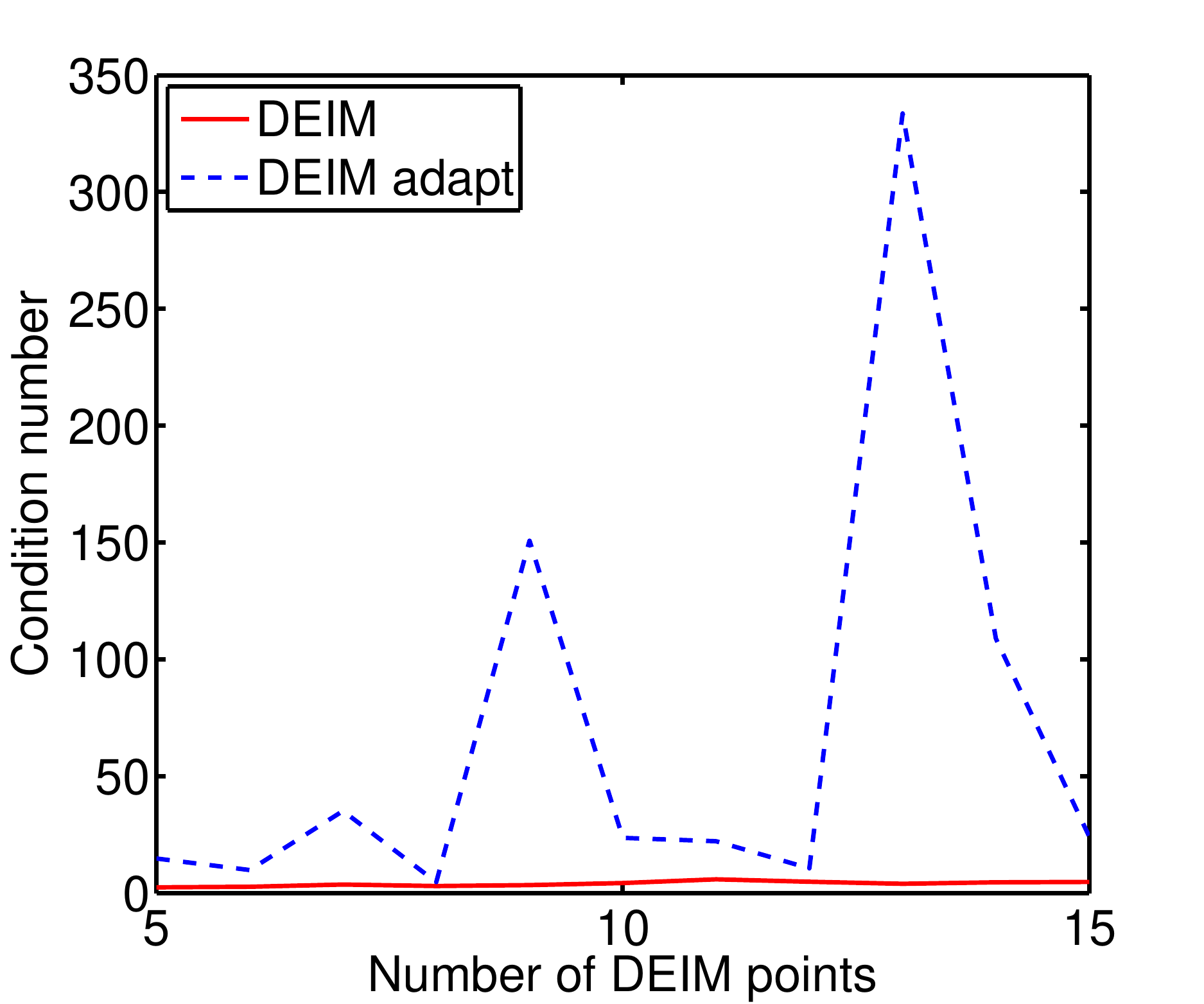}}
\caption{\label{fig::error_global_nonlinear_cond_number_Burgers} Comparison between traditional and adaptive DEIM strategies - Global non-linear term error at time step $t_2$ in the Euclidian norm (left panel); Condition number of matrix $P^TV$. }
\end{figure}

Figure \ref{fig::error_global_nonlinear_cond_number_Burgers}(a) illustrates the comparison between the 1D-Burgers advection term errors computed at time $t_2$ using the standard vs. adaptive DEIM methods. The number of points is varied and the results show that standard DEIM approximation is more accurate over the global spatial domain $[0,1]$. This is expected since the adaptive DEIM method is tailored to improve the accuracy of the quantity of interest and not the global accuracy. The condition number of the matrix $P^TV_{\mu}$ \eqref{eqn::POD_DEIM_nonlinearity} is increased in the case of adaptive DEIM approximaton, contributing to the global loss of accuracy. This result is shown in Figure \ref{fig::error_global_nonlinear_cond_number_Burgers}(b).

The adaptive DEIM approximation of the non-linear term leads to a different reduced order solution. The quantity of interest is computed using POD/DEIM standard and adaptive models and its errors are presented in Figure \ref{fig::adaptive_DEIM_efficiency_Burgers}. The adaptive model leads to more accurate quantities of interest confirming the usefulness of our a-posteriori error estimation results and adaptive DEIM algorithm.

\begin{figure}[t!]
\centering
\includegraphics[scale=0.37]{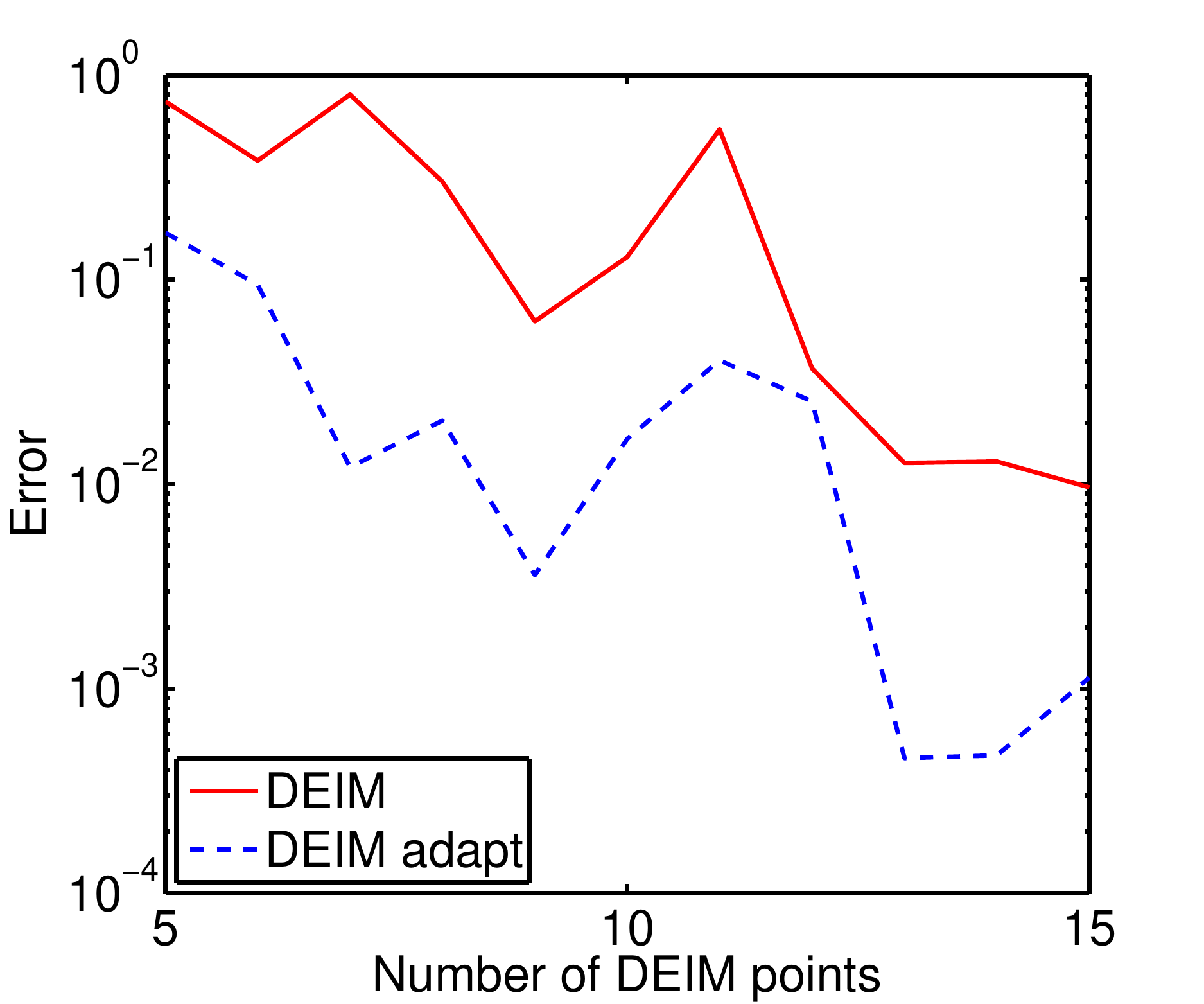}
\caption{\label{fig::adaptive_DEIM_efficiency_Burgers} Adaptive vs traditional DEIM errors approximation errors of the quantity of interest.}
\end{figure}


\subsection{Shallow Water Equations (SWE) model}\label{sec:SWE}

The SWE is a popular simple model for meteorological and oceanographic problems. SWE can be used to model Rossby and Kelvin waves in the atmosphere, rivers, lakes and oceans as well as gravity waves in a smaller domain. The alternating direction fully implicit (ADI) scheme \citep{Gus1971} considered in this paper is first order in both time and space and it is stable for large CFL condition numbers. It has been proven that the method is unconditionally stable for the linearized version of the SWE model. Other research work on this topic include efforts of \citet{FN1980} and \citet{NVG1986}.

We solve the SWE model using the $\beta$-plane approximation on a rectangular domain \citep{Gus1971}
\begin{equation}\label{eqn:swe-pde}
\frac{\partial w}{\partial t}=A(w)\,\frac{\partial w}{\partial x}+B(w)\,\frac{\partial w}{\partial y}+C(y)\,w,
\quad (x,y) \in [0,L] \times [0,D], \quad t\in(0,t_{\rm f}],
\end{equation}
where $w=(u,v,\phi)^T$ is a vector function, and $u,v$ are the velocity components in the $x$ and $y$ directions, respectively. The geopotential $\phi = 2\sqrt{gh}$ is computed by multiplying the depth of the fluid $h$ and the acceleration due to gravity  $g$.

The matrices $A$, $B$ and $C$ are
\[
A=-\begin{bmatrix}
           u&0&\phi/2\\
           0&u&0\\
           \phi/2&0&u \end{bmatrix}, \quad
B=-\begin{bmatrix}
           v&0&0\\
           0&v&\phi/2\\
           0&\phi/2&v \end{bmatrix}, \quad
C=\begin{bmatrix}
           0&f&0\\
           -f&0&0\\
           0&0&0 \end{bmatrix},
\]
where $f$ is the Coriolis term
\[
f=\hat f + \beta(y-D/2),\quad \beta=\frac{\partial f}{\partial y},
\]
with $\hat f$ and $\beta$ constants.

We assume periodic solutions in the $x$ direction for all three state variables.
In the $y$ direction
\[
v(x,0,t)=v(x,D,t)=0, \quad x\in[0,L], \quad t\in(0,t_{\rm f}],
\]
and Neumann boundary conditions are used for $u$ and $\phi$. The initial condition is $w(x,y,0)=\psi(x,y),~\psi:\mathbb{R}\times\mathbb{R}\rightarrow \mathbb{R},~(x,y)\in[0,L]\times[0,D]$.

The space discretization is performed on a uniform mesh with $n = N_x\cdot N_y$ equidistant points on $[0,L]\times[0,D]$, with $\Delta x=L/(N_x-1),~\Delta y=D/(N_y-1)$. We also discretize the time interval $[0,t_{\rm f}]$ using $N_t$ equally distributed points and $\Delta t=t_{\rm f}/(N_t-1)$. The discrete solution vector is:
\[
{\boldsymbol w}(t_N)\approx [w(x_i,y_j,t_N)]_{i=1,2,\ldots,N_x,~j=1,2,\ldots,N_y} \in \mathbb{R}^{n}, \quad N=1,2,\ldots,N_t.
\]

After spatial discretization of  \eqref{eqn:swe-pde} the semi-discrete SWE equations are:
\begin{equation}
\label{discrete_SWE}
\begin{split}
 {\bf u}' & =  -F_{11}({\bf u},\bm{{\phi}})-F_{12}({\bf u},{\bf v}) + {\bf F}\odot {\bf v}, \\
  {\bf v}' & =  -F_{21}({\bf u},{\bf v})-F_{22}({\bf v},\bm{{\phi}}) - {\bf F}\odot {\bf u}, \\
  {\bm{{\phi}}}' & =  -F_{31}({\bf u},\bm{{\phi}})-F_{32}({\bf v},\bm{{\phi}}),
  \end{split}
\end{equation}
where $\odot$ is the component-wise multiplication operator, ${\bf u}'$, ${\bf v}'$, ${\bm{{\phi}}}'$ denote semi-discrete time derivatives, and ${\bf F} ={\bf f} \cdot \mathbf{1}_{N_x}^T$ stores Coriolis components ${\bf f} = [f(y_j)]_{j=1,2,..,N_y}$. The non-linear terms involving derivatives in $x$ and $y$ directions, respectively, are defined as follows:
\begin{eqnarray*}
&& F_{i1},F_{i2}: \mathbb{R}^{n} \times \mathbb{R}^{n} \rightarrow \mathbb{R}^{n},~i=1,2,3, \\
&&  F_{11}({\bf u},{\boldsymbol \phi})={\boldsymbol u}\odot A_x{\boldsymbol u}+\frac{1}{2}{\boldsymbol \phi}\odot A_x{\boldsymbol\phi},
\quad  F_{12}({\boldsymbol u},{\boldsymbol v})={\boldsymbol v}\odot A_y{\boldsymbol u}, \\
  &&F_{21}({\boldsymbol u},{ \boldsymbol v})={\boldsymbol u}\odot A_x{\boldsymbol v}, \quad F_{22}({\boldsymbol v},{\boldsymbol\phi})={\boldsymbol v}\odot A_y{\boldsymbol v}+\frac{1}{2}{\boldsymbol \phi}\odot A_y{\boldsymbol\phi}, \\
  &&F_{31}({\boldsymbol u},{\boldsymbol \phi})=\frac{1}{2}{\boldsymbol\phi} \odot A_x {\boldsymbol u}+{\boldsymbol u} \odot {A_x\boldsymbol \phi}, \\
  && F_{32}({\boldsymbol v},{\boldsymbol \phi})=\frac{1}{2}{\boldsymbol\phi}\odot A_y{\boldsymbol v}+{\boldsymbol v} \odot A_y{\boldsymbol\phi}.
\end{eqnarray*}

Here $A_x,A_y\in \mathbb{R}^{n\times n}$ are constant coefficient matrices for discrete first-order and second-order  differential operators which take into account the boundary conditions.

The numerical scheme is implemented in Fortran and uses a sparse matrix environment. For operations with sparse matrices we employ the SPARSEKIT library \citep{Saad1994}, and the sparse linear systems obtained during the quasi-Newton iterations are solved using MGMRES library \citep{Barrett94,Kelley95,Saad2003}.  Here we do not decouple the model equations like in  \citet{Stefanescu2013} where the Jacobian is either block cyclic tridiagonal or block tridiagonal.

\subsection{Numerical experiments with the shallow water model}\label{sec:numerical_SWE}

We computed the initial conditions from the initial height condition No. 1 of \citet{Gram1969} i.e.
\begin{equation*}
h(x,y,0)=H_0+H_1+\tanh\biggl(9\frac{D/2-y}{2D}\biggr)+H_2\textrm{sech}^2\biggl(9\frac{D/2-y}{2D}\biggr)\sin\biggl(\frac{2\pi x }{L}\biggr).
\end{equation*}
The initial velocity fields are derived from the initial height field using the geostrophic relationship

\[
u = \biggl(\frac{-g}{f}\biggr)\frac{\partial h}{\partial y}, \quad v = \biggl(\frac{g}{f}\biggr)\frac{\partial h}{\partial x}.
\]

We use the following constants $ L=6000km,~D=4400km,~\hat f=10^{-4}s^{-1}$, $\beta=1.5\cdot10^{-11}s^{-1}m^{-1},~g=10 m s^{-2},~H_0=2000m,~H_1=220m,~H_2=133m.$

The domain is discretized using a mesh of  $31\times17 = 527$ points, with $\Delta x = 200$km and $\Delta y = 275$km. We select the integration time window to be $24$h and we use
$181$ time steps corresponding to $\Delta t = 480$s.

The considered quantity of interest depends on some particular components of the geopotential $\phi$ at the final
time step and it is defined below
\begin{equation}\label{eqn:SWE_qoi}
 \mathcal{Q}(\phi) = \sum_{i=1}^{6} \sum_{j=2}^{8}\phi(x_i,y_j,t_{N_t}),~[x_1,x_{6}] \times [y_2,y_8] = [0,1000]\textrm{km } \times [275,1925]\textrm{km}.
\end{equation}

To generate the reduced order forward and adjoint SWE models, we collect $180$ snapshots from each high-fidelity model representing the state solutions $u$, $v$ and $\phi$ and their corresponding adjoint variables at every time step. For each state variable, a snapshot matrix with $360$ columns containing both forward and adjoint solutions is formed. Three POD bases are then derived following singular value decompositions. Moreover, snapshots of the six forward non-linear terms were also collected and $6$ POD bases required by the DEIM approximation were constructed. Finally, we apply a Galerkin projection to the discrete SWE model and the forward POD/DEIM reduced order model is derived. More details regarding the construction of the SWE reduced order model can be found in \cite{stefanescu2014comparison}. Applying chain derivatives, the reduced adjoint model is obtained following the ARRA method \cite{cstefuanescu2015pod}.

Next we asses the a-posteriori error estimation result \eqref{eqn::aposteriori1} using numerical experiments. The quantity of interest \eqref{eqn:SWE_qoi} is computed using the forward high-fidelity and reduced order models and the true error is obtained. Using the projected solution of the reduced of adjoint model the estimated error of the quantity of interest is computed. Comparative results are shown in Figure \ref{fig::apost_same_config_SWE}. In panel (a), we set the number of DEIM points and test various POD bases dimensions. The estimated errors are very accurate even for smaller bases dimensions. Panel (b) describes a different experiment where the dimensions of POD bases are fixed to $30$ and the number of DEIM points is varied. Again, the estimated errors shows good agreement with the true.

\begin{figure}[t!]
  \centering
  \subfigure[Number of DEIM points $=35$] {\includegraphics[scale=0.3]{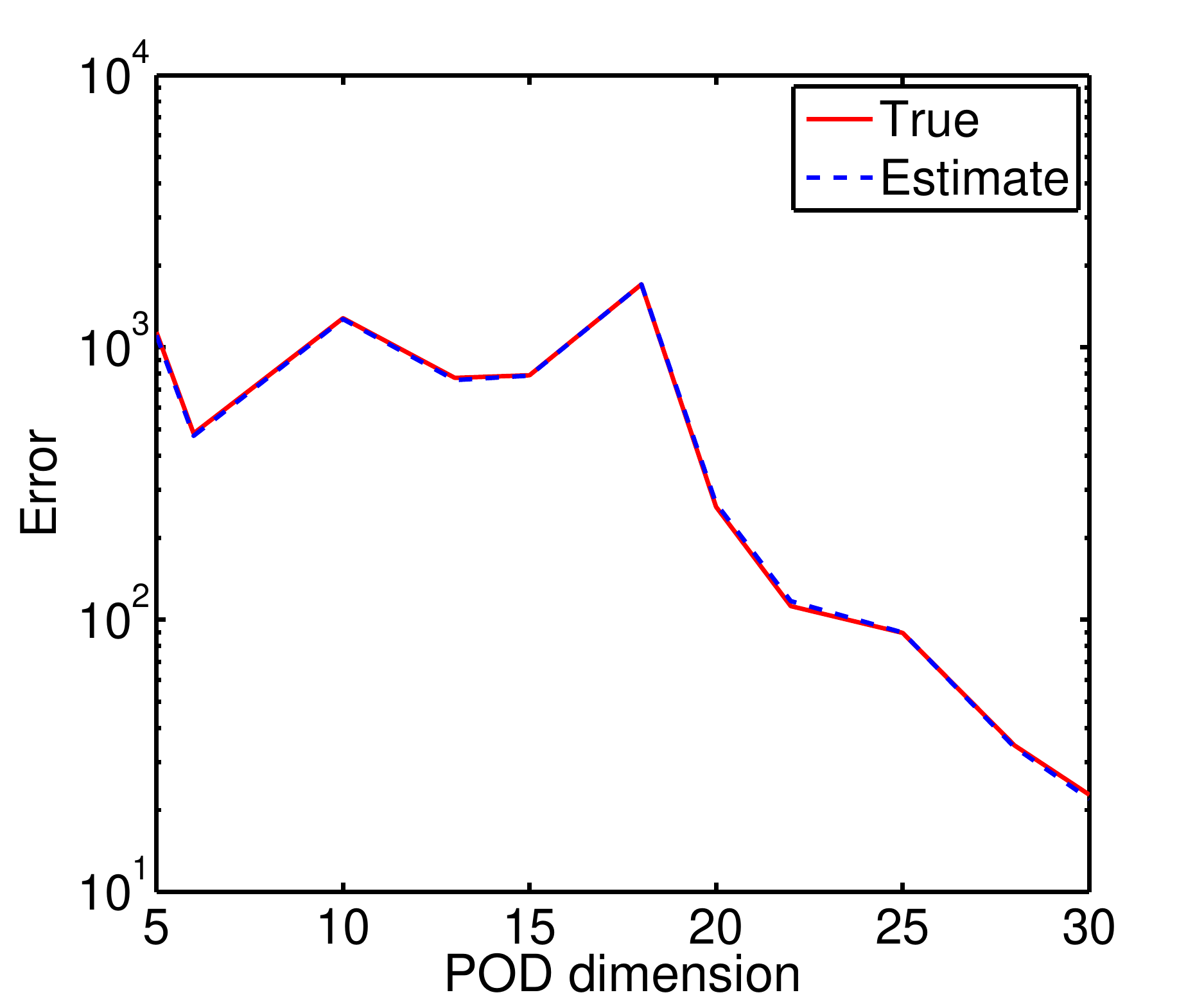}}
  \subfigure[Dimension of POD basis $=30$ ]{\includegraphics[scale=0.3]{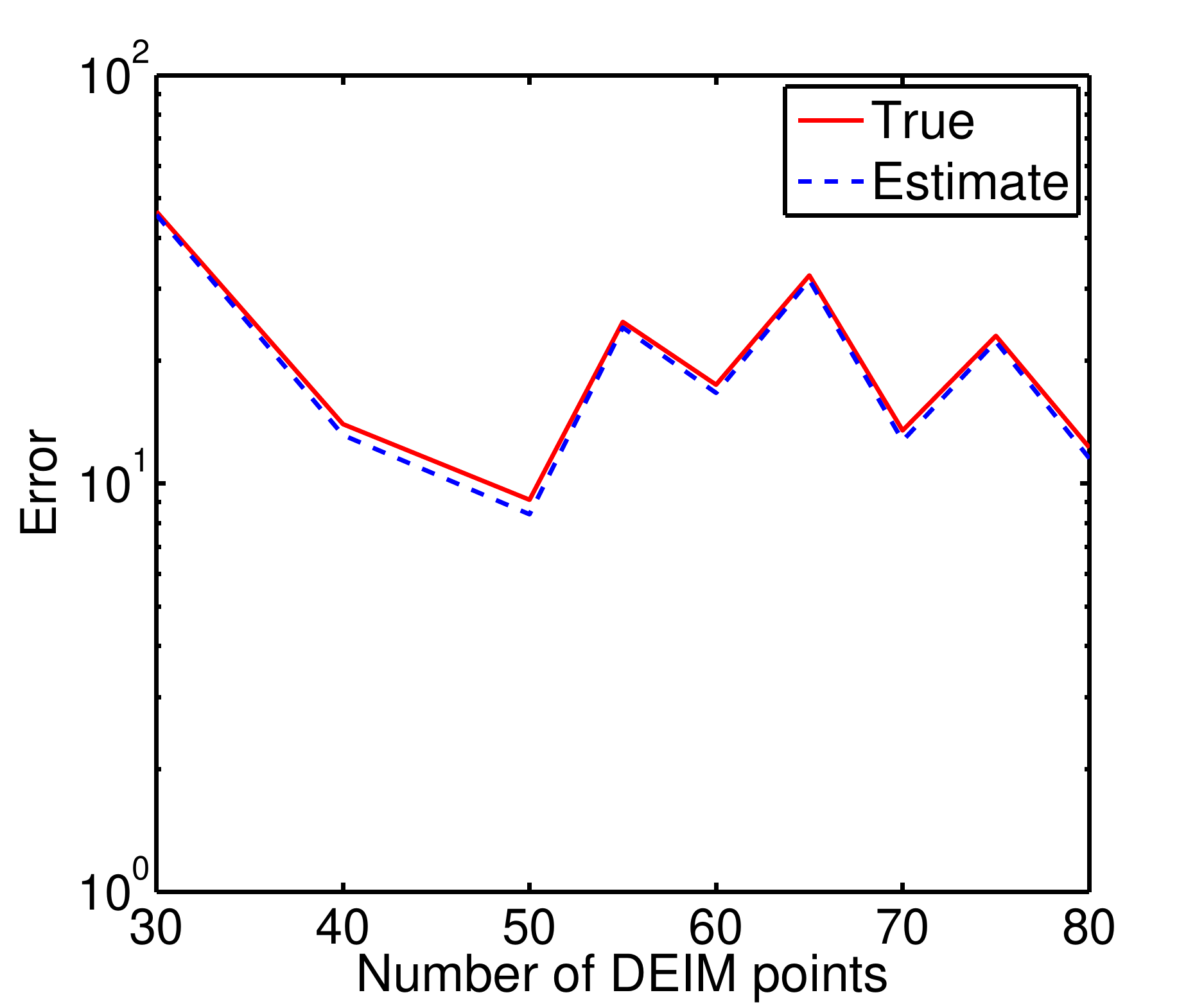}}
\caption{\label{fig::apost_same_config_SWE}A-posteriori error estimates for the same parametric configuration.}
\end{figure}

In the sequel, we make use of the dual weighted residuals employed by the a-posteriori error estimation result \eqref{eqn::aposteriori1} to generate adaptive DEIM points and decrease the quantity of interest error computed using the reduced order model. The dual weighted residuals are first computed replacing the dot product with the component-wise multiplication in the right-hand side part of the expression \eqref{eqn::aposteriori1}; i.e., $ {\widehat{\mathbf \lambda}}_i^T \cdot \Delta \x_i,~i=0,\ldots,N_t$. These residuals are then used together with the non-linear bases to generate the adaptive DEIM points. For example, the continuity equation dual weighted residuals are employed together with the non-linear basis of the non-linear term $F_{31}$ and the resulted adaptive DEIM points are depicted in Figure \ref{fig::DEIM_points_SWE}. We notice that more points are placed in the spatial domain of the quantity of interest delimited by the red rectangle in comparison with the standard DEIM points. Parameter $\alpha$ required by Algorithm \ref{alg::DEIM_adaptiv} was set to $\frac{1}{2}$.

\begin{figure}[t!]
\centering
\includegraphics[scale=0.37]{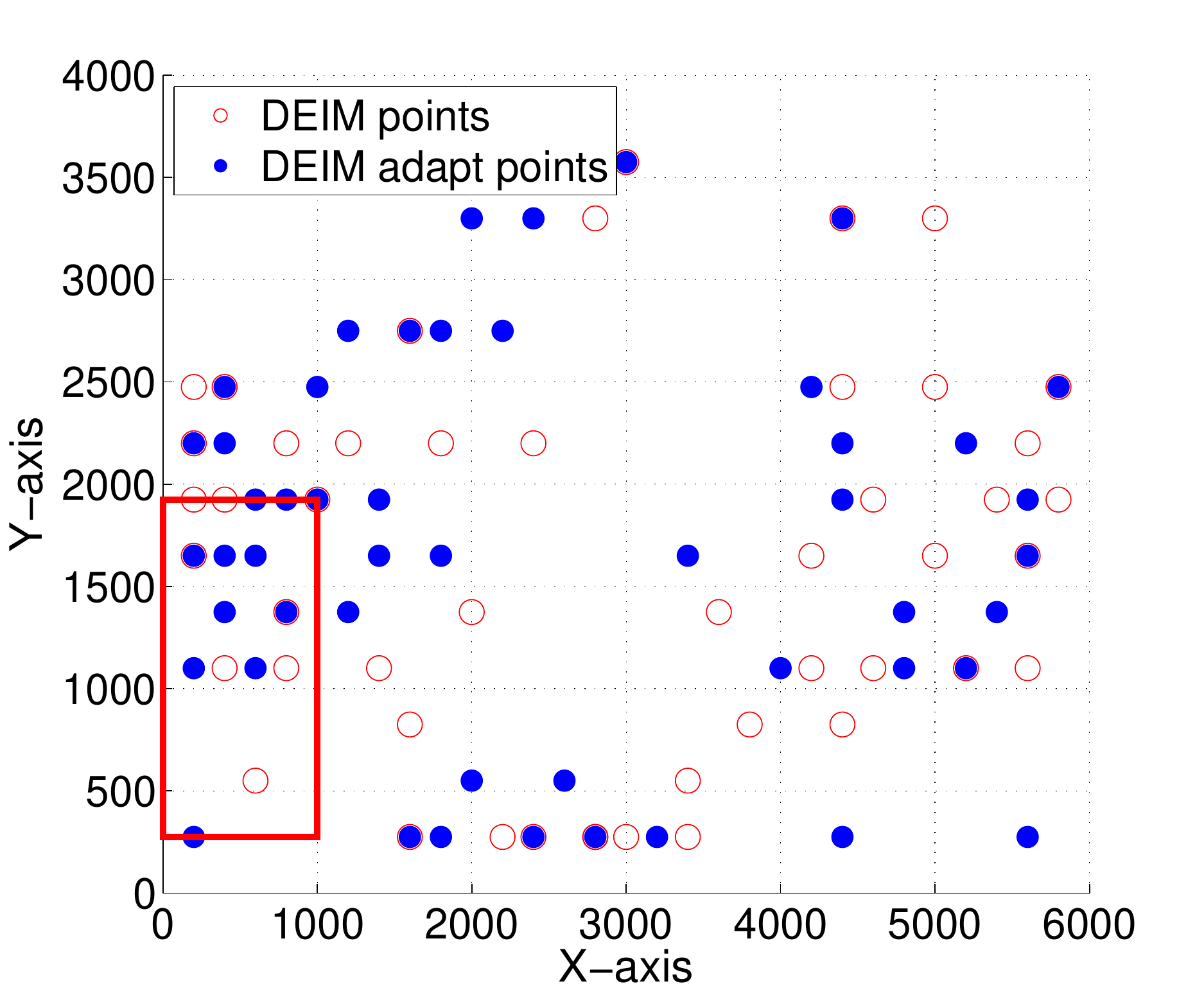}
\caption{\label{fig::DEIM_points_SWE} Adaptive vs traditional DEIM points.}
\end{figure}

The adaptive DEIM algorithm changes the DEIM approximation of the non-linear terms since matrix $P$ \eqref{eqn::POD_DEIM_nonlinearity} is modified. This leads to modified accuracy of the non-linear terms approximations over all the spatial points. The errors associated with various DEIM approximations of the non-linear terms $F_{31}$ at time step $t_2$ are shown in Figure \ref{fig::error_global_nonlinear_cond_number_SWE} (a). The adaptive DEIM approximation is less accurate than the standard DEIM over the entire spatial domain. This is explained by larger condition numbers of $P^TV_{\mu}$ \eqref{eqn::POD_DEIM_nonlinearity}  generated by the adaptive DEIM algorithm as noticed in Figure \ref{fig::error_global_nonlinear_cond_number_SWE} (b).


\begin{figure}[t!]
  \centering
  \subfigure[Non-linear approximation] {\includegraphics[scale=0.3]{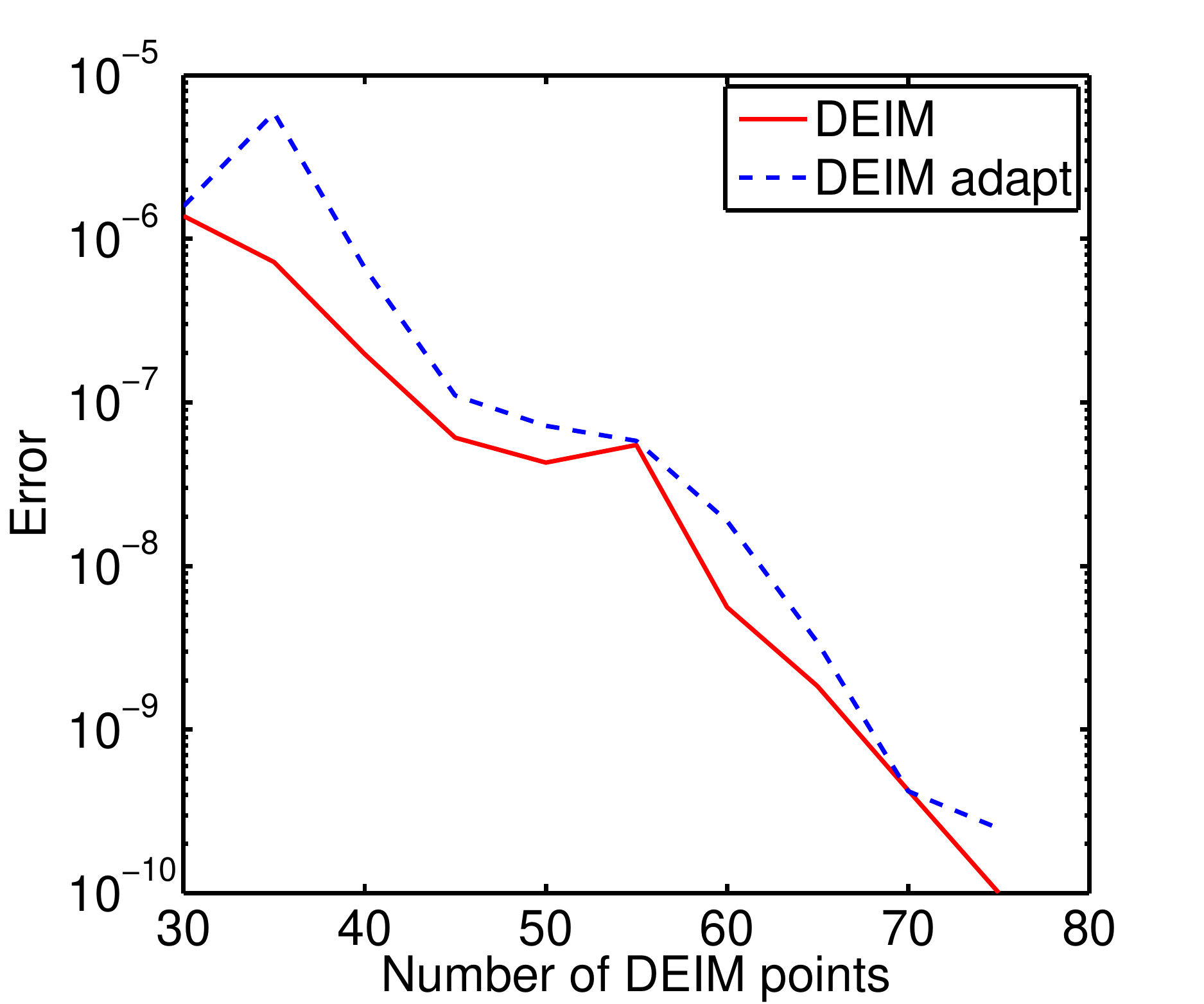}}
  \subfigure[Condition numbers]{\includegraphics[scale=0.3]{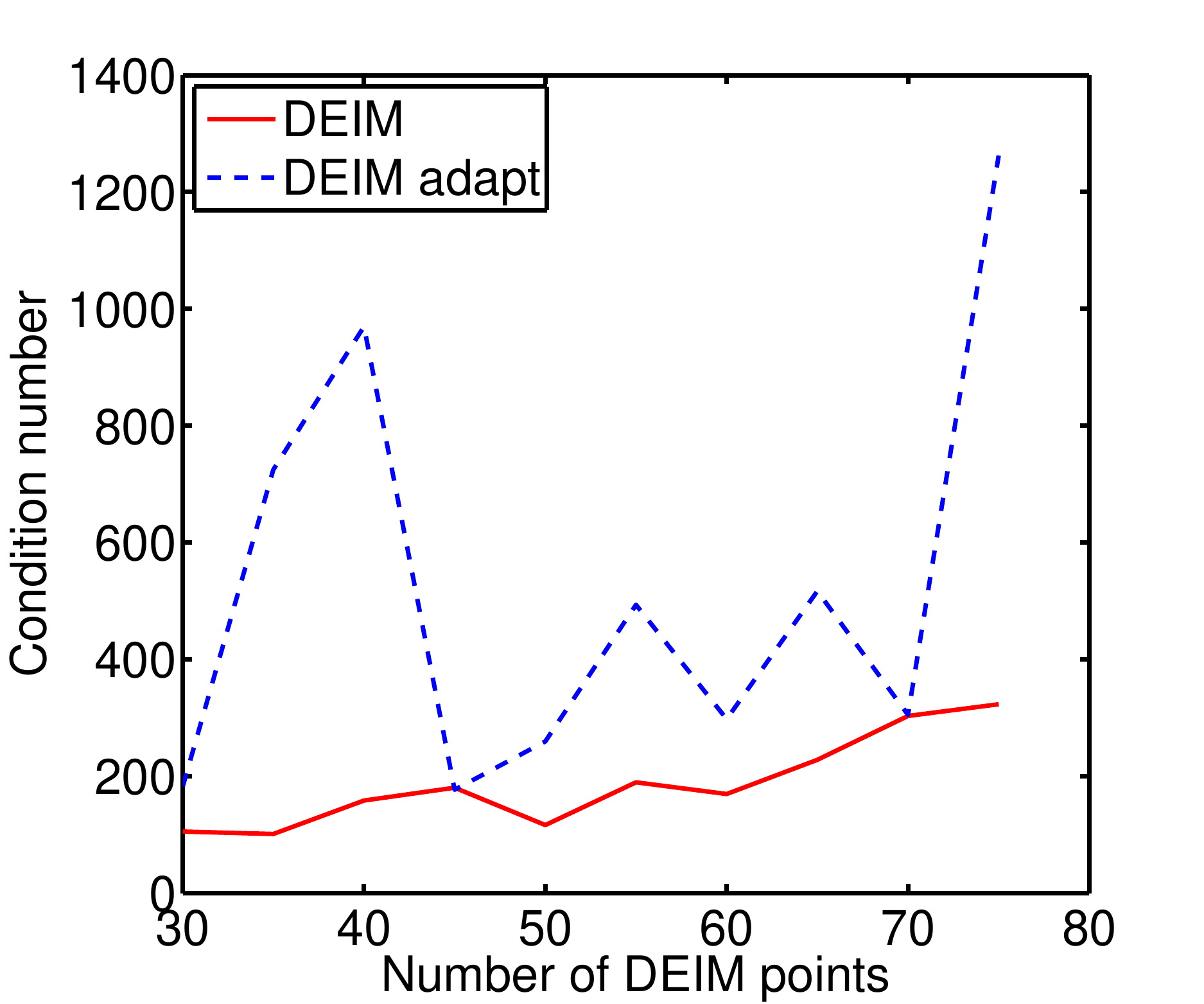}}
\caption{\label{fig::error_global_nonlinear_cond_number_SWE} Comparison between traditional and adaptive DEIM strategies - Global non-linear term error at time step $t_2$ in the Euclidian norm (left panel); Condition number of matrix $P^TV$ for non-linear term $F_{31}$. }
\end{figure}

\begin{figure}[t!]
  \centering
  \subfigure[Dimension of POD basis $=10$] {\includegraphics[scale=0.3]{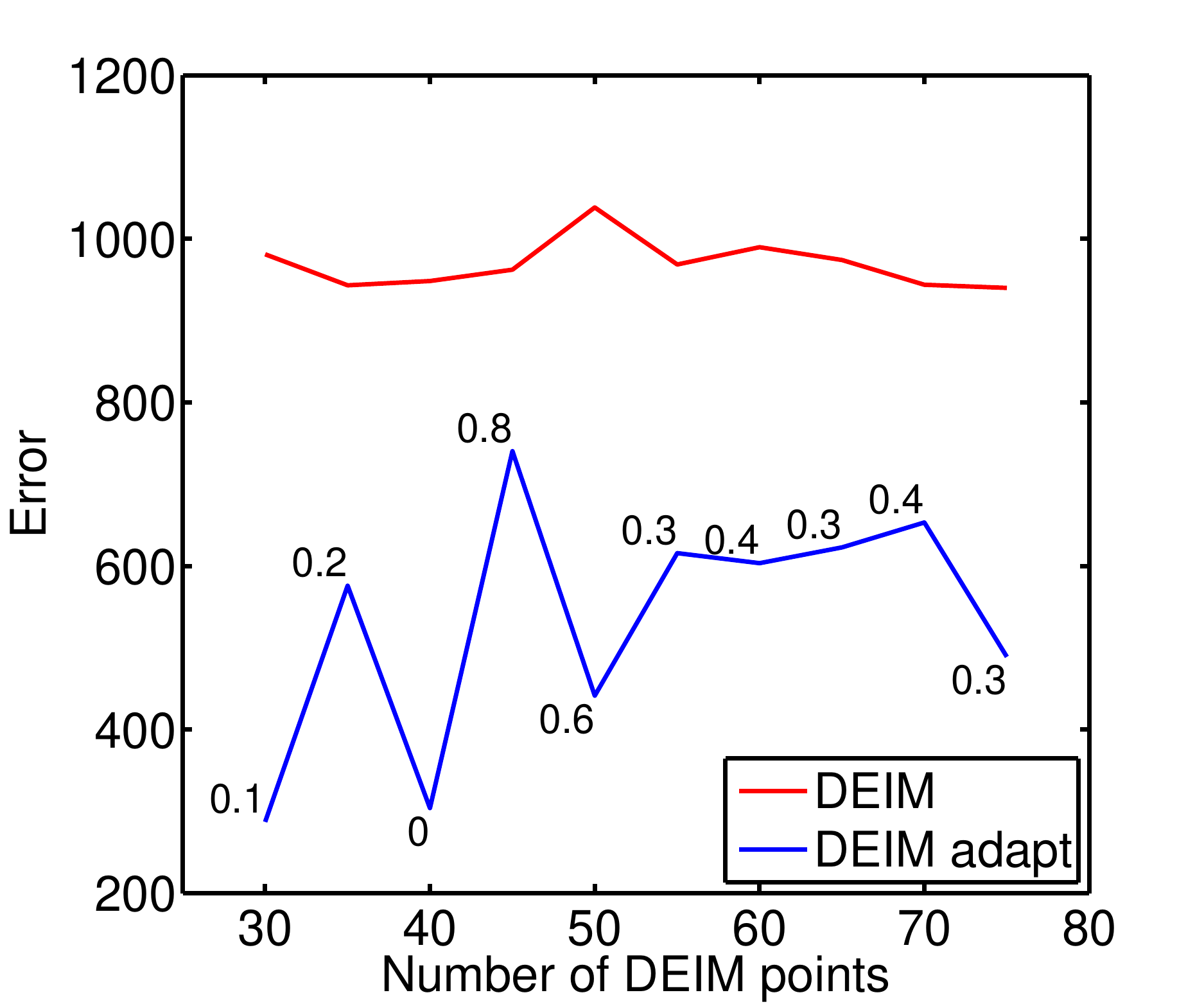}}
  \subfigure[Dimension of POD basis $=20$]{\includegraphics[scale=0.3]{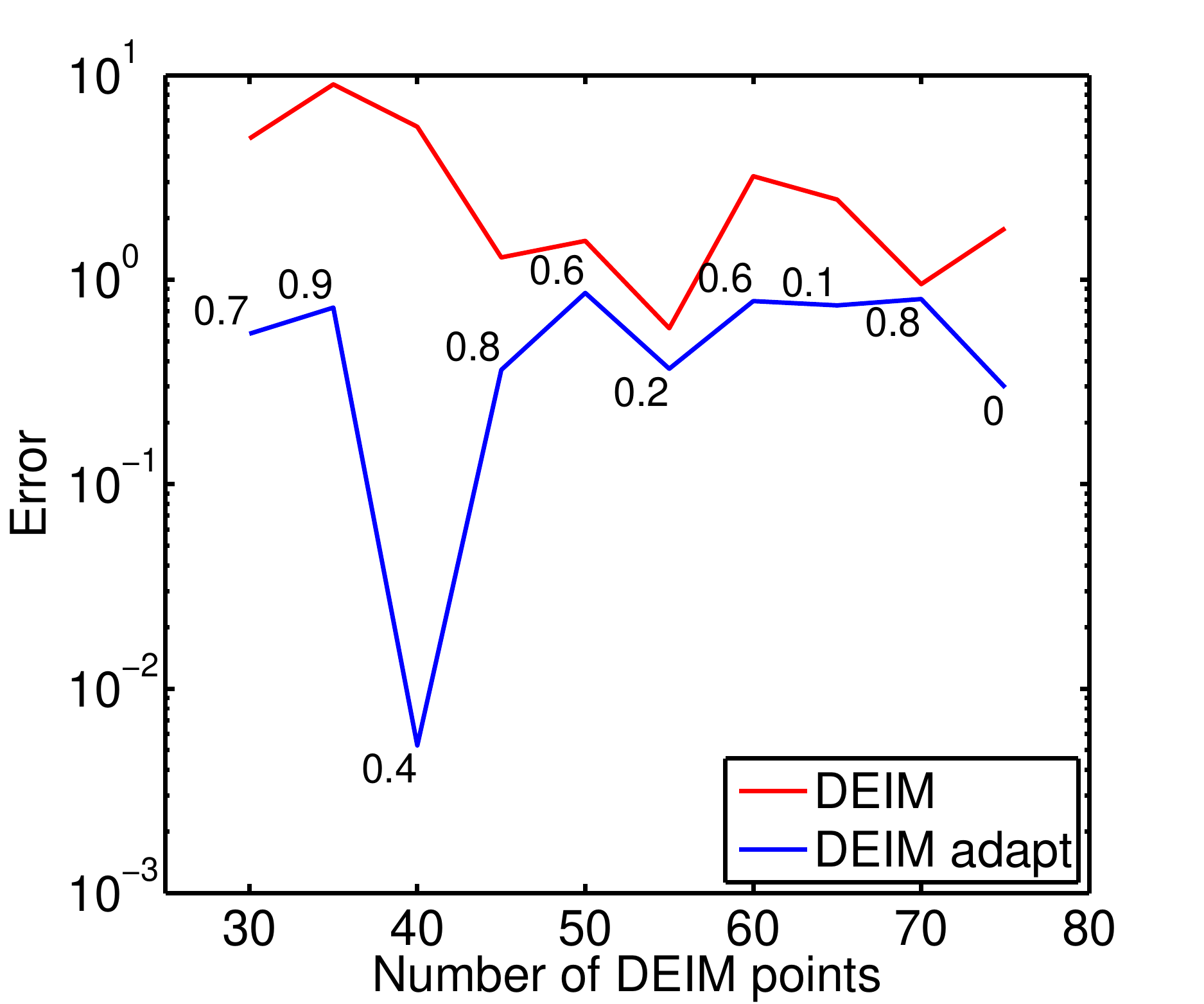}}
\caption{\label{fig::adaptive_DEIM_efficiency_SWE} Adaptive vs traditional DEIM errors approximation errors of the quantity of interest.  }
\end{figure}

Next we compare the quantity of interest errors obtained using POD/DEIM and POD/adaptive DEIM reduced order models. Among $11$ values of $$\alpha=\{0,0.1,0.2,\ldots,0.9,1\},$$ we show only the resulted most accurate estimates in Figure \ref{fig::adaptive_DEIM_efficiency_SWE} for two POD bases dimensions. The labels along the adaptive DEIM trajectories describe the values of $\alpha$ that generated the most accurate quantities of interest. In general, it is more difficult to enhance the quantity of interest accuracy once the reduced order state solution is very precise. This is obvious by comparing the results presented in Figure \ref{fig::adaptive_DEIM_efficiency_SWE}. Larger POD basis dimensions usually leads to more accurate ROM solutions.

\section{Conclusions}\label{sec:conclusions}

In this paper we proposed a-posteriori error estimation results to estimate the errors in quantities of interests calculated using POD/DEIM reduced order models. Later, using the dual weighted residuals, we introduced an adaptive DEIM algorithm to improve the accuracy of these quantities of interests.

The first a-posteriori error estimation result sums up the first-order expansions of the quantity of interest's components computed as the product between the gradient of each component and high-fidelity model residuals. The gradient of each component is obtained as the solution of a high-fidelity adjoint model. Efficient versions of this a-posteriori error estimation result are obtained for the explicit and implicit Euler schemes by assuming a good level of accuracy for the surrogate models solutions. Only one reduced forward and adjoint model runs and the evaluation of the high-fidelity model residuals are required for estimating the quantity of interest errors for both schemes. Reduced order models generated using a single parametric configuration were employed to estimate the quantity of interest for a different parametric configuration in the case of 1D-Burgers model. In the case of the SWE model, we tested only the same parametric configuration case. Several POD basis dimensions and numbers of DEIM points were tested and the estimated error accuracy increased with more accurate reduced order model solutions.

Next we designed an adaptive DEIM algorithm to increase the precision of the quantity of interest computed using the forward reduced order model. The singular vectors of the dual weighted residuals together with the non-linear term basis are used to generate the new DEIM points. This changes the DEIM approximation of the non-linear terms and subsequently the associated reduced order model solutions. Numerical results using the new reduced order models revealed that the accuracy of the quantities of interest was enhanced for both 1D-Burgers and SWE models.

In the future we plan to apply regression machine learning models to predict the coefficient $\alpha$ used to combine the dual weighted residuals singular vectors and non-linear basis within the adaptive DEIM algorithm. In combination with MP-LROM models \cite{moosavi2017multivariate}, the a-posteriori error estimation result and the adaptive DEIM algorithm will provide accurate outcomes along the entire parametric space.

\paragraph{\bf Acknowledgment}
The work of Dr. R\u azvan Stefanescu and Prof. Adrian Sandu was supported by the NSF CCF--1218454, AFOSR FA9550--12--1--0293--DEF, AFOSR 12-2640-06, and by the Computational Science Laboratory at Virginia Tech. R\u azvan \c Stef\u anescu would like to thank Vishwas Rao for useful conversations about error estimation results.



\bibliographystyle{abbrvnat}
\bibliography{Bib/Software,Bib/ROM_state_of_the_art,Bib/POD_bib,Bib/CDS_E_proposal,Bib/sandu,Bib/comprehensive_bibliography1,Bib/Razvan_bib,Bib/Razvan_bib_ROM_IP,Bib/NSF_KB,Bib/data_assim_weak-fdvar,Bib/reduced_models}

\end{document}

%% file: notation_Razvan.tex
\newcommand{\Ns}{N_{\rm{state}}}
\newcommand{\Nt}{N_t}
\newcommand{\tx}{ \bf \tilde x}
\newcommand{\bx}{\mathbf x}